\documentclass[12pt]{amsart}
\usepackage{amsfonts,graphicx,verbatim,amssymb,amsmath,multirow,bbold}
\usepackage[all]{xy}
\usepackage[active]{srcltx}
\usepackage{hyperref}

\DeclareSymbolFont{bbold}{U}{bbold}{m}{n}
\DeclareSymbolFontAlphabet{\mathbbold}{bbold}

\newtheorem{thm}{Theorem}[subsection]
\newtheorem{cor}[thm]{Corollary}
\newtheorem{lem}[thm]{Lemma}
\newtheorem{prop}[thm]{Proposition}

\theoremstyle{definition}
\newtheorem{dfn}[thm]{Definition}
\theoremstyle{remark}
\newtheorem{rem}[thm]{Remark}
\newtheorem{ex}[thm]{Example}

\def\A{\mathbb{A}}

\def\C{\mathbb{C}}

\def\R{\mathbb{R}}

\def\Z{\mathbb{Z}}
\def\1{\mathbbold{1}}

\def\Cc{\mathcal{C}}

\def\Fc{\mathcal{F}}

\def\Ic{\mathcal{I}}

\def\Sc{\mathcal{S}}

\def\ef{\mathfrak{e}}

\def\gf{\mathfrak{g}}
\def\lf{\mathfrak{l}}

\def\qf{\mathfrak{q}}

\def\hf{\mathfrak{h}}

\def\al{\alpha}
\def\be{\beta}
\def\ga{\gamma}
\def\eps{\varepsilon}
\def\e{\varepsilon}
\def\la{\lambda}

\def\ta{\theta}
\def\si{\sigma}
\def\om{\omega}
\def\de{\delta}

\def\d{\partial}
\def\0{\varnothing}
\def\sm{\setminus}
\def\ol{\overline}

\def\le{\leqslant}
\def\ge{\geqslant}
\def\<{\langle}
\def\>{\rangle}

\def\el{\mathrm{EL}}

\def\uc{\mathbb{S}} 

\def\disk{\mathbb{D}}

\def\mod{\mathrm{mod}}

\def\Len{\mathrm{Len}}
\def\diam{\mathrm{diam}}
\def\area{\mathrm{area}}

\def\ch{\mathrm{CH}}

\def\Hmes{\mathrm{Hmes}} 
\def\Hdim{\mathrm{Hdim}} 

\def\th{\mathrm{Th}}

\def\Arg{\mathrm{Arg}}

\begin{document}

\title[Equipotential annuli and the Hausdorff measure]{Renormalization, equipotential annuli and the Hausdorff measure}

\author[A.~Blokh]{Alexander~Blokh}

\thanks{The first named author was partially supported by NSF grant DMS-2349942.}


\thanks{The 
second named author was partially
supported by NSF grant DMS--1807558}


\author[L.~Oversteegen]{Lex Oversteegen}

\author[V.~Timorin]{Vladlen~Timorin}

\thanks{The 
third named author was supported by
the HSE University Basic Research Program.
}

\begin{abstract}
For a complex single variable polynomial $f$ of degree $d$, let $K$ be its \emph{filled Julia set}, i.e., the union of all bounded orbits. Assume that $K$ has an invariant component $K^*$ on which $f$ acts as a degree $d_*<d$ map. This is a simplest instance of \emph{holomorphic polynomial-like renormalization} (Douady--Hubbard). One can associate a certain Cantor-like subset $G'$ of the circle with $K^*$; it is defined as the set of arguments of all smooth or broken rays to $K^*$. We will describe a role the Hausdorff dimension of $G'$ and the respective Hausdorff measure play in geometry of $K^*$. In particular, we give upper and lower bounds on the modulus of renormalization in terms of the Hausdorff measure of $K^*$.
\end{abstract}

\maketitle

\section{Introduction}
\label{s:intro}
In this paper, we consider one of the simplest instances of holomorphic polynomial-like renormalization \cite{DH85}.
Namely, let $f:\C\to\C$ be a degree $d$ polynomial such that the filled Julia set $K(f)$
has an invariant component $K^*$ with $f|_{K^*}$ of degree $d_*$, where $1<d_*<d$.
As is well known, $K^*$ is then a polynomial-like filled Julia set for a suitable restriction of $f$.

On can take a Jordan disk $U^*\supset K^*$ bounded by an equipotential
such that $\d U^*$ contains at least one critical point of $f$ while all critical points inside $U^*$ lie in $K^*$.
Such $U^*$ defines a polynomial-like map $f:U^*\to f(U^*)$, see Fig. \ref{fig:Kstar}.
Our main objective is to estimate the modulus of $U^*\sm K^*$ from above and below
in terms of the value $\eps$
of the Green function on $\d U^*$. While a lower bound can be obtained by a straightforward application of the Gr\"otzsch inequality,
an upper bound is much more subtle. The answer will involve the Hausdorff dimension and the corresponding Hausdorff measure
of a certain subset of the circle associated with $K^*$.

\begin{figure}
  \centering
  \includegraphics[width=.8\textwidth]{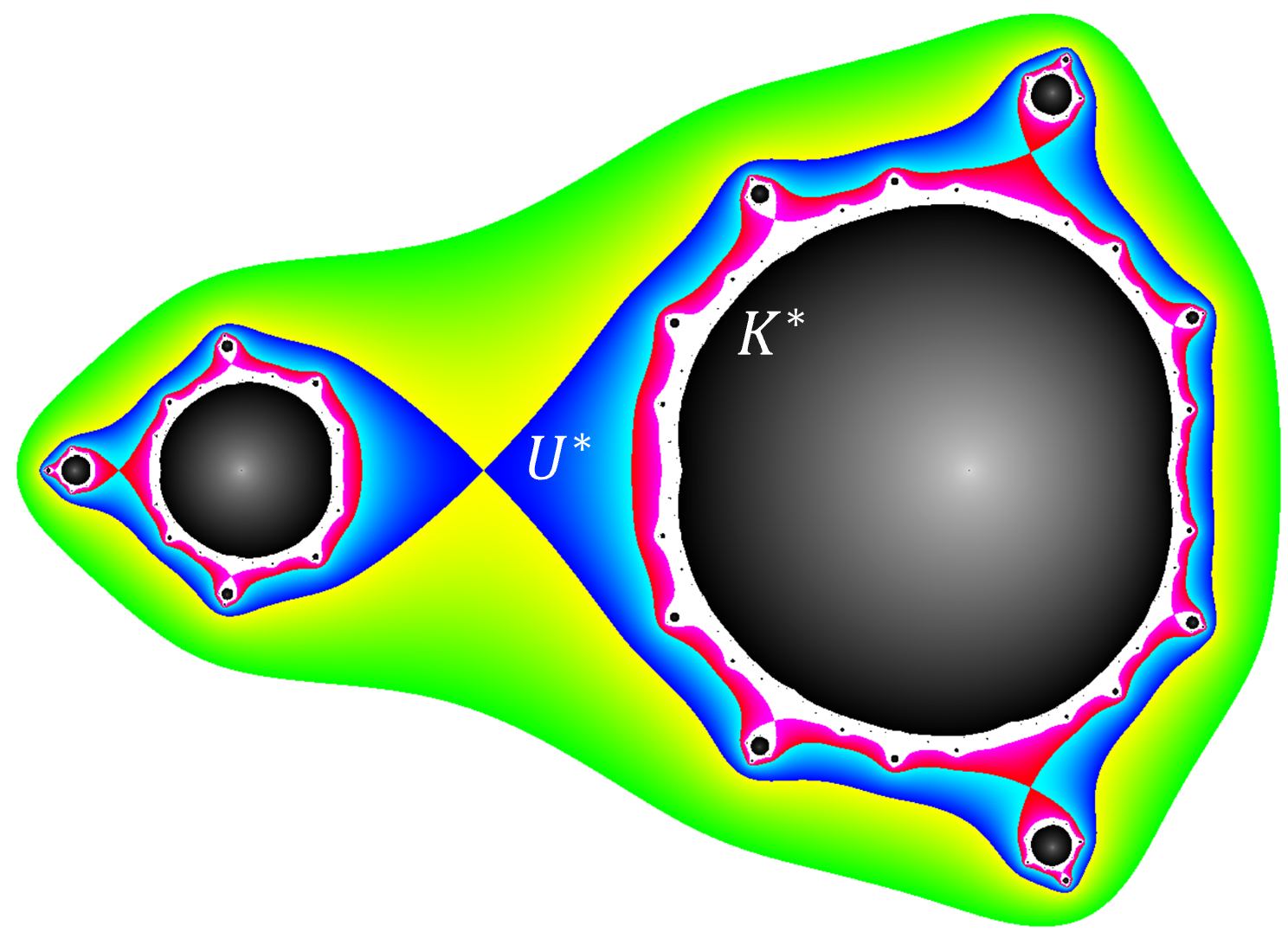}
  \caption{The dynamical plane of a cubic polynomial $f$, for which the PL map $f:U^*\to f(U^*)$
  is hybrid equivalent to $z\mapsto z^2$.}\label{fig:Kstar}
\end{figure}

Both an upper and a lower bound on $\mod(U^*\sm K^*)$ appeared in DeMarco--McMullen \cite[Theorem 4.4]{DM08}.
However, their main interest was in $\eps\to\infty$ (where $\eps$ is the equipotential level 
 of $\d U^*$), while we are more interested in the asymptotics as $\eps\to 0$.
Bounds of \cite{DM08} are $C\eps \le \mod(U^*\sm K^*) \le C\eps + C'$, 
 with the same $C$ on both sides, but with a positive additive error term $C'$ in the upper estimate.
This upper bound fails to asymptotically vanish as $\eps\to 0$, while we want not only to see 
 that $\mod(U^*\sm K^*)\to 0$ as $\eps\to 0$, which is straightforward, but also to have an estimate on the convergence rate.
Our inequalities have the form $\underline{C}\eps\le \mod(U^*\sm K^*)\le \ol C\eps$, where
the coefficients $\underline{C}$ and $\ol C$ do not depend on $\eps$ but depend
on the ``combinatorics'' of $f|_{K^*}$.
A more precise meaning of this is explained below.

\subsection{Equipotentials and rays}
Here is a brief and informal discussion, focusing on geometric and physical intuition;
 the reader is referred to Section \ref{s:cgapdata} for more formal definitions and to \cite{Bro65,DH84} as original sources.
A compactum $K\subset\C$ defines the \emph{Green function} $g_K:\C\to\R_{\ge 0}$ that vanishes precisely on $K$,
 is harmonic on $\C\sm K$, and has a logarithmic singularity at infinity
 (more precisely, $g_K(z)-\log|z|$ is bounded near infinity).

When $K=K(f)=\{z\in\C\mid f^n(z)\not\to \infty\}$ for a complex polynomial $f:\C\to \C$,
 the Green function $g_f:=g_K$ of $K$ acquires a dynamical meaning.
Namely, $g_f$ measures the escape rate: the larger $g_f(z)$ is, the faster $z$ escapes
 to infinity under the iterates of $f$ (observe that points of $K$ do not escape, which
 is consistent with $g_f=0$ on $K$).
Level curves of $g_f$ are called \emph{equipotentials}, and orthogonal trajectories are called \emph{external intervals}.
\emph{External rays}, i.e., unbounded external intervals, can be labeled by \emph{angles} $\ta\in\R/\Z$ as $R_f(\ta)$
 so that $f(R_f(\ta))\subset R_f(d\ta)$, where $d$ is the degree of $f$.
The angle $\ta$ is then called the \emph{external argument} of the ray $R_f(\ta)$.

\subsection{Polynomial-like renormalization}
A \emph{polynomial-like (PL) map} \cite{DH85} is a proper holomorphic map $F:V\to U$ of degree $\deg(F)>1$,
 where $V$ and $U$ are Jordan domains such that $\ol V$ is compact and contained in $U$.
One is usually interested in dynamics of $F$ only near its (PL) \emph{filled Julia set}
$$
K(F)=\{z\in V\mid f^n(z)\in V,\ \forall n\in\Z_{\ge 0}\}.
$$
Polynomials give the motivating example of PL maps: given a polynomial $P$,
 define $U$ as a sufficiently large disk, set $V:=P^{-1}(U)$, and $F:=P|_V$.
However, many non-polynomial maps admit PL restrictions:
 for example, adding a uniformly small holomorphic map on $U$ to $F$ while simultaneously
 replacing $U$ with a suitable smaller domain yields another PL map close to $F$.
Also, polynomials of large degree often admit PL restrictions of smaller degree.
If $Q_c(z):=z^2+c$, then degree two PL restrictions of the iterates of $Q_c$
 play a key role in the proof \cite{DH85} that the Mandelbrot set contains many homeomorphic copies of itself.
More generally, passing from a polynomial $P$ to a PL restriction of an iterate $P^n$ of $P$
 is a powerful method in holomorphic dynamics known as \emph{PL renormalization}.

What makes PL renormalization work is the following \emph{straightening theorem} of Douady and Hubbard \cite{DH85}:
 every PL map $F$ is topologically conjugate near $K(F)$ to a polynomial.
In fact, this topological conjugacy $h$ can be made quasi-conformal on a neighborhood of $K(F)$ and
 satisfy $\ol\d h=0$ (in the sense of distributions) on $K(F)$; such conjugacies are called \emph{hybrid}.
A polynomial hybrid conjugate to a PL map $F$ with connected $K(F)$ is unique, up to affine change of variables.
The space of PL maps (with a suitable topology; see \cite{Sul88,Lyu99}) is non-compact; however, the only way a variable PL map $F$ can go to ``infinity''
 (i.e., to leave any compact part of the space) is when the \emph{fundamental annulus} $U\sm\ol V$ degenerates.

Recall that any topological annulus $A$ (an open subset of $\C$ homeomorphic to a cylinder)
 is conformally isomorphic to a unique Euclidean cylinder of circumference one and height $\mu$;
 the number $\mu$, called the \emph{modulus} of $A$ and denoted by $\mod(A)$, is therefore
 a conformal invariant of $A$.
Having $\mod(U\sm\ol V)$, the so called \emph{renormalization modulus}, bounded below,
 for a variable PL map $F:V\to U$, means keeping $F$ in a compact part of the space of all PL maps.
For this reason, lower bounds on $\mod(U\sm \ol V)$, known as \emph{a priori bounds for $F$}, are important.
A priori bounds play a key role in the Feigenbaum--Coullet--Tresser universality \cite{Sul88,Lyu99,McM96},
 local connectedness of the Mandelbrot set at ``most'' points \cite{Hub93,Mil00a}, its universal appearance in
 analytic parameter spaces \cite{McM00}, etc.

Upper bounds on the modulus $\mod(U\sm\ol V)$ measure how fast $F$ escapes to infinity
 in certain families of PL maps; such upper bounds are a central theme of this paper.
Note that upper bounds on the \emph{root annulus} $\mod(U\sm K(F))$ are stronger
 than the upper bounds on the corresponding \emph{fundamental annulus} $U\sm\ol V$ (the former imply the latter);
 for this reason, we mostly consider the root annulus.
A prototypical example of how an upper bound of the renormalization modulus
implies some interesting dynamical facts is presented in \cite{BOT24}.

\subsection{Combinatorics of $f|_{K^*}$}
\label{ss:combK*}
Identify $\R/\Z$ with the unit circle $\uc=\{z\in\C\mid |z|=1\}$ by the map $\ta\mapsto e^{2\pi i\ta}$.
The (Riemannian) metric on $\R/\Z$ is induced from the standard Euclidean metric on $\R$;
 lengths will refer to this metric.

Recall our basic setup: $f:\C\to\C$ is a degree $d$ polynomial such that $K(f)$ contains
 an invariant component $K^*$, on which $f$ is (generically) $d_*$-to-one;
 we assume that $1<d_*<d$.
Then there is $\eps=\eps_f>0$ with the property that the equipotential $\{g_f=\eps\}$
 contains a Jordan curve bounding a disk $U^*$ such that $f$ has no critical points in $U^*\sm K^*$
 while the boundary $\d U^*$ of $U^*$ contains at least one critical point.
Note that $\eps$ depends not only on $f$ but also on the choice of $K^*$.
However, since the latter choice will always be clear, we omit $K^*$ from the notation.
It is straightforward and well known that $f:U^*\to f(U^*)$ is polynomial-like of degree $d_*$.

Let $G'$ be the set of arguments of all external rays or broken rays that accumulate in $K^*$;
 the set $G'$ consists of the \emph{perfect part} $G'_p$, which is a Cantor set,
 and the at most countable set $G'\sm G'_p$.
Complementary components of $G'$ in $\uc$ are called \emph{holes} of $G'$.
Even though $G'$ has countably many holes, it suffices to know finitely many holes to recover $G'$.
Namely, every hole of length $<\frac{1}{d}$ is mapped by the $d$-tupling map $\si_d:x\mapsto dx\pmod \Z$
 to another hole; hence, the finitely many holes of length $\ge \frac{1}{d}$ determine $G'$.

Let $\om_+$ be the fastest escaping (i.e., the closest to infinity in terms of $g_f$) critical point of $f$.
By definition, $g_f(\om_+)\ge \eps$.
We will need the nonnegative integer $m=m_f$ with the property $d^m\eps\le g_f(\om_+)<d^{m+1}\eps$.
Our estimates of $\mod(U^*\sm K^*)$ will depend only on $G'$ and $m$.

\subsection{Hausdorff measure and dimension of $G'$}
It is not hard to see that $\si_d:G'\to G'$ is generically $d_*$-to-one.
A combinatorial part of this paper deals with geometry of $G'$.

\begin{thm}
\label{t:main-comb}
The set $G'$ has Hausdorff dimension $\delta_*:=\log_d d_*$.
Moreover, the $\delta_*$-Hausdorff measure $\Hmes^{\delta_*}(G')$ of $G'$ satisfies the inequalities
$$
\frac 1{d_*^2}\le \Hmes^{\delta_*}(G')\le \frac{d-d_*}{d_*-1} d_*^{\delta_*}.
$$
\end{thm}

Previous results in this direction are due to Petersen and Zakery \cite{PZ22}:
they proved that $\delta_*$ is an upper bound on the Hausdorff dimension of $G'$
while we prove that in fact the Hausdorff dimension of $G'$ \emph{equals} $\delta_*$.
Our approach is hands-on and elementary;
simultaneously with proving that the Hausdorff dimension of $G'$ is $\delta_*$
it gives bounds on the $\delta_*$-Hausdorff measure of $G'$.

\subsection{Bounds on the modulus of $U^*\sm K^*$}
The main
 part of the paper establishes inequalities on $\mod(U^*\sm K^*)$.

\begin{thm}
\label{t:main-an}
Let $f$ be a degree $d$ polynomial, and $f:U^*\to f(U^*)$ be a degree $d_*$
 polynomial-like map as above.
If $m>0$, then
$$
\mod(U^*\sm K^*)\ge \left(\frac{d}{d-1}\right)^{m+1}\hspace{-6pt}\cdot\frac{\eps}{2\pi\,d_*}
$$
 and, on the other hand,
$$
\mod(U^*\sm K^*) \le
\frac{d^{3m+2}}{d_*^{3m-2}}\cdot \frac{(d-d_*)\eps}{2\pi\cdot\Hmes^{\delta_*}(G')^2}\le
\frac{d^{3m+2}(d-d_*)\eps}{2\pi\cdot d_*^{3m-6}}.
$$
where $\eps=\eps_f$ and $m=m_f$ are as defined in Section \emph{\ref{ss:combK*}}.
\end{thm}

Note that a stretching deformation applied to $f$ changes $\eps$ (multiplies it by the stretching factor)
but does not affect the parameters $\underline{C}$, $\ol C$ of the inequalities
 $\underline{C}\eps\le \mod(U^*\sm K^*)\le \ol C\eps$.
It looks interesting and relevant that the upper bound involves the Hausdorff measure $\Hmes^{\delta_*}(G')$;
 this Hausdorff measure is directly related with the extremal length of curves in $U^*$ looping once around $K^*$.
The Hausdorff measure is also relevant for the local version of the upper bound
 from Theorem \ref{t:main-an}, namely, when, instead of $U^*\sm K^*$, one
 considers a topological quadrilateral cut from the annulus $U^*\sm K^*$ by
 a pair of rays to $K^*$; see Lemma \ref{l:len-qf}.

The lower bound on $\mod(U^*\sm K^*)$ confirms that some noteworthy phenomena may take place.
Namely, it is possible that $\eps\to 0$ while, simultaneously, $m\to \infty$ so that potentially
$\mod(U^*\sm K^*)\to\infty$, say, in some one-parameter family of polynomials $f_\eps$.

In the formulas of Theorem \ref{t:main-an}, some simplifications are
 made to shorten the outcome at the expense of precision;
 more precise versions follow from the techniques of this paper.
Also, the assumption $m>0$ is used in the process of simplification,
 so that the result is \emph{not} applicable to $m=0$.
Since the case $m=0$ is of special interest, we state it separately.
For example, this case realizes when all escaping critical points have the same escape rate,
 in particular, when there is only one escaping critical point.

\begin{thm}
\label{t:m=0}
If $m=0$ in the setting of Theorem $\ref{t:main-an}$, then
$$
\frac{d(d-1)\eps}{2\pi d_*^2(d_*-1)}\le \mod(U^*\sm K^*)\le \frac{d(d-1)(d-d_*)d_*^{2\delta_*+6}\eps}{2\pi(d_*-1)^2}.
$$
\end{thm}

For example, suppose that $d_*=d-1$; then $f$ has only one escaping critical point.
The inequalities of Theorem \ref{t:m=0} reduce in this special case to
$$
\frac{d\eps}{2\pi (d-1)(d-2)}\le \mod(U^*\sm K^*)\le \frac{d(d-1)^{2\delta_*+7}\eps}{2\pi(d-2)^2}.
$$

In particular, if $d=3$ and $d_*=2$,
$$
\frac{\eps}{5}<\frac{3\eps}{4\pi}\le \mod(U^*\sm K^*)\le \frac{384\cdot 4^{\log_32}\cdot\eps}\pi<147\eps.
$$

We hope that the upper bound on $\mod(U^*\sm K^*)$,
which is the main topic of this paper,
can be extended to the case when $U^*$ is a topological disk with no escaping critical points in $U^*$,
 but $U^*$ is not necessarily a domain of a PL map and is not bounded by an equipotential.

\subsection{Extremal distance between rays}
\label{ss:ED}
Given a domain $\Omega\subset \C$ and two sets $A$, $B$ in the plane,
 define the \emph{extremal distance} $\mathrm{dist}_\Omega(A,B)$ between $A$ and $B$ with respect to $\Omega$
 as the extremal length of the collection $\Fc_\Omega(A,B)$ of all
 piecewise smooth arcs in $\Omega$ connecting $A$ with $B$ (cf. \cite{ahl66}).
More precisely, a piecewise smooth arc $\gamma\subset\Omega$ that is a homeomorphic image of $\R$ in $\Omega$
 belongs to $\Fc_\Omega(A,B)$ if $\gamma$ accumulates in $A$ and $B$ at the two ends of $\gamma$
 (if $\gamma$ is parameterized by $\R$, the latter means $\gamma(t)\to A$ as $t\to -\infty$ and $\gamma(t)\to B$
 as $t\to +\infty$).
Given a circle arc $(\al;\be)\subset\uc$ such that $\al$, $\be\in G'$,
 define the \emph{sector} $\Sigma(\al;\be)$
 as the complementary component of the union of $K^*$ with the smooth or broken rays $A$, $B$ to $K^*$
 of arguments $\al$, $\be$ that contains external rays $R(\ga)$ with $\ga\in (\al;\be)$.
Let $\mathrm{ED}(\al;\be)$ be the extremal distance between $A=R^*_f(\al)$ and $B=R^*_f(\be)$
 with respect to $U^*\cap \Sigma(\al;\be)$.
The extremal distance $\mathrm{ED}(\al;\be)$ can be estimates in terms of the
 $\delta_*$-Hausdorff measure of $G'\cap [\al;\be]$ as follows.

\begin{thm}
\label{t:ED}
For a circle arc $(\al;\be)\subset\uc$ such that $\al$, $\be\in G'$,
$$
\mathrm{ED}(\al;\be)\ge C\cdot \frac{\Hmes^{\delta_*}(G'\cap [\al;\be])^2}{\eps}
$$
 where $C$ depends on $d$, $d_*$, and $m$, but not on $\al$, $\be$, or $\eps$.
\end{thm}

Making $C$ explicit is straightforward and similar to Theorems \ref{t:main-an} and \ref{t:m=0}.
For example, one can take
$$
C=\begin{cases}
    \frac{2\pi d_*^{3m-2}}{d^{3m+1} (d-1)(d-d_*)}, & \mbox{if } m>0 \\
    \frac{2\pi (d_*-1)^2}{d(d-1)d_*^{2+2\delta_*}(d-d_*)}, & \mbox{if } m=0
  \end{cases}
$$
 (see Section \ref{ss:upper}).
While the Hausdorff measure can be easily eliminated from Theorem \ref{t:main-an},
 its appearance on the right-hand side of inequalities from Theorem \ref{t:ED} is essential.

\subsection{Organization of the paper}
In Section \ref{s:cgapdata}, we recall some terminology of complex polynomial dynamics,
 specifically, the notions related to the geometry and dynamics of the basin of infinity
 (various kinds of rays, equipotentials, etc.).
We proceed with studying the relationship between $(f,K^*)$ and a certain
 invariant gap $G$, called the \emph{renormalization gap of $(f,K^*)$}.
Geometry and dynamics of $(f,K^*)$ define a sequence of approximations $\Cc_n\subset\uc$ of $G':=G\cap\uc$
 (here, $n$ ranges in $\Z$, and $\Cc_n$ converges to $G'$ as $n\to -\infty$).
Each $\Cc_n$ is a union of finitely many closed circle arcs called \emph{level $n$ connectors}.
Also, a certain tiling of $\C\sm K^*$ by nested annuli $A_n$, called the
 \emph{fundamental nest}, is described.

Section \ref{s:metr} deals with combinatorial estimates related to $G$ but not (explicitly) to $K^*$.
We start (in Section \ref{ss:holconn}) by estimating the number of connectors, their total length, as well as
 the length of the longest connector of any given level.
Next, a map $\psi$ is studied that collapses all holes of $G'$ to points and semi-conjugates $\si_{d}|_{G'}$ with $\si_{d_*}$.
H\"older continuity of $\psi$ implies a lower bound on the Hausdorff measure of $G'$
 while an upper bound can be obtained by definition, using the results of Section \ref{ss:holconn}.
Theorem \ref{t:HdimG} proved in Section \ref{s:Hol} implies Theorem \ref{t:main-comb}.
As $G'$ is approximated by $\Cc_n$, the Hausdorff measure on $G'$, extended by zero to all of $\uc$,
 is approximated with measures $\lf_n$ supported on $\Cc_n$; the corresponding convergence
 results are established in Section \ref{ss:meas}.
Lemma \ref{l:Hmes-Scn} is a key geometric fact about $G$ that is used
 when estimating $\mod(U^*\sm K^*)$; it relates the Hausdorff measure of $[\al;\be]\cap G'$
 with the $\lf_n$-measure of $[\al;\be]$; in particular, it implies a lower bound on $\lf_n[\al;\be]$
 that is independent of $n$.

Section \ref{s:geom-est} contains lower and upper bounds on $\mod(U^*\sm K^*)$.
While the lower bound (Section \ref{ss:lowb}) 
follows from known techniques
 (Gr\"otzsch inequality), the upper bound is based on extremal length estimates for
  a specific, carefully chosen, conformal metric on $U^*\sm K^*$.
This conformal metric $\qf$ is introduced in Section \ref{ss:pf}, after a general discussion
  of conformal metrics (Section \ref{ss:metrA}) and a geometric definition of the set
  on which $\qf$ is supported (Section \ref{ss:rect}).
Section \ref{ss:len-qf} contains length and area estimates with respect to the metric $\qf$,
 which yield the main results of this paper (proved in Section \ref{ss:upper}).

 \subsection{Acknowledgments}
 The authors would like to thank G. Levin for numerous fruitful discussions and useful suggestions.

\section{Renormalization gap}
\label{s:cgapdata}
Consider a monic centered complex degree $d$ polynomial $f$ with disconnected Julia set;
 $f$ being \emph{monic centered} means that $f(z)=z^d+a_{d-2}z^{d-2}+\dots +a_0$ for some $a_{d-2}$, $\dots$, $a_0\in\C$.
The \emph{Green function} for $f$ is defined as
$$
g_f(z):=\lim_{n\to\infty} \frac{\log_+ |f^n(z)|}{d^n},\quad\hbox{where}\quad \log_+(z):=\max(0,\log |z|).
$$
As is well known, the basin of infinity with respect to $f$ is the set of all points $z$ with $g_f(z)>0$,
 and $g_f$ is harmonic on this basin.
Critical points of $g_f$ in the basin of infinity are precisely eventual preimages of escaping critical points of $f$;
 we call these points \emph{singularities of $g_f$} in order not to overload the term ``critical point''.

We refer the reader to \cite{Mil00,DM08,PZ22} for general discussion of dynamical Green functions and,
 more generally, complex polynomial dynamics in the disconnected case.

\subsection{External rays and equipotentials}
\label{ss:exteq}
By an \emph{external interval}, we mean a complete flow line (=trajectory) of the vector field $-\nabla g_f$,
where the gradient is with respect to the usual Euclidean metric $|dz|$ on $\C$. More precisely,
``complete'' here means extended to all times (gradient flow lines come with parameterization,
and complete flow lines are parameterized by all of $\R$). Thus, a complete flow line may be asymptotic
to $\infty$, to $K(f)$, or to a singularity of the potential.
An external interval either extends all the way from infinity to $K(f)$, in which case it is called a \emph{smooth external ray}, or is asymptotic to a singularity of $g_f$
in one or both directions.
Singularities of the gradient vector field $\nabla g_f$ coincide with the singularities of $g_f$.
An \emph{external ray} is an unbounded external interval; it is either a smooth
external ray or a trajectory of $-\nabla g_f$ starting at infinity that is asymptotic to a singularity.
In the latter case, we may say that an external ray \emph{hits} the singularity,
or \emph{crashes into} the singularity although, formally speaking, the corresponding
trajectory never reaches it in finite time.
External rays are uniquely determined by their arguments;
 we will write $R_f(\ta)$ (sometimes omitting the $f$ from the notation)
 for the external ray of $f$ of argument $\ta$.

\begin{dfn}[External arguments]
By definition, any point of $R_f(\ta)$ has \emph{external argument} $\ta$.
If $z\in\C\sm K(f)$ is not in an external ray, then $z$ is approximated by points of $R_f(\ta_n)$,
 for a suitable sequence $\ta_n\in\uc$.
Any partial limit of $\ta_n$ is called an \emph{external argument} of $z$.
Write $\Arg(z)$ for the set of all external arguments of $z$;
 this set may consist of more than one angle.
Given any set $Z\subset \C\sm K(f)$, let $\Arg(Z)$ be the union of the sets $\Arg(z)$ over all $z\in Z$.
\end{dfn}

If, say, exactly two external rays $R(\al)$, $R(\be)$ hit a singularity $z$ of $g_f$,
 then $\Arg(z)=\{\al,\be\}$.
Also, for any subset $Z$ of a bounded external interval incident to $z$, we have $\Arg(Z)=\{\al,\be\}$.

\begin{dfn}[Paths and broken rays to $K^*$]
By a \emph{simple path} in an open set $W\subset \C$ we mean the image of $\R$ under an 
 embedding $\gamma:\R\to W$ such that
$\lim_{t\to -\infty} \gamma(t)=\infty$.
If $\Gamma$ is a simple path and $\ol{\Gamma}\sm \Gamma\subset Y$, then we say that $\Gamma$ is a \emph{path to $Y$}.
By a \emph{path to $K^*$} we mean a simple path to $K^*$ in $\C\sm K(f)$.
External rays that are paths to $K^*$ are called \emph{external rays to $K^*$}.
Also, a simple path $R$ from infinity to $K^*$ is called a \emph{broken ray to $K^*$} (cf. \cite{BCLOS16}) if
 $R$ consists of at most countably many external intervals and singularities of $g_f$.
Define the \emph{external argument} of $R$ as the only angle $\ta\in\R/\Z$ such that $R_f(\ta)\subset R$
(it is well-known that the external argument of a broken ray is well-defined).
\end{dfn}

\emph{Equipotential curves}, or \emph{equipotentials}, are defined as level curves of
 the Green function $g_f$ for $f$.
Connected components of equipotentials are called \emph{equipotential ovals}
 if they are Jordan curves.
See Fig. \ref{fig:rays} for an illustration of external intervals and an equipotential curve
 (the boundary of the entire figure)
in the case of a cubic polynomial $f$ with a quadratic-like PL restriction.
Also, recall that the \emph{topological hull} $\th(X)$ of a compactum $X$ in $\C$ is defined as
the complement of the unbounded component of $\C\sm X$.

\begin{figure}
  \centering
  \includegraphics[width=.8\textwidth]{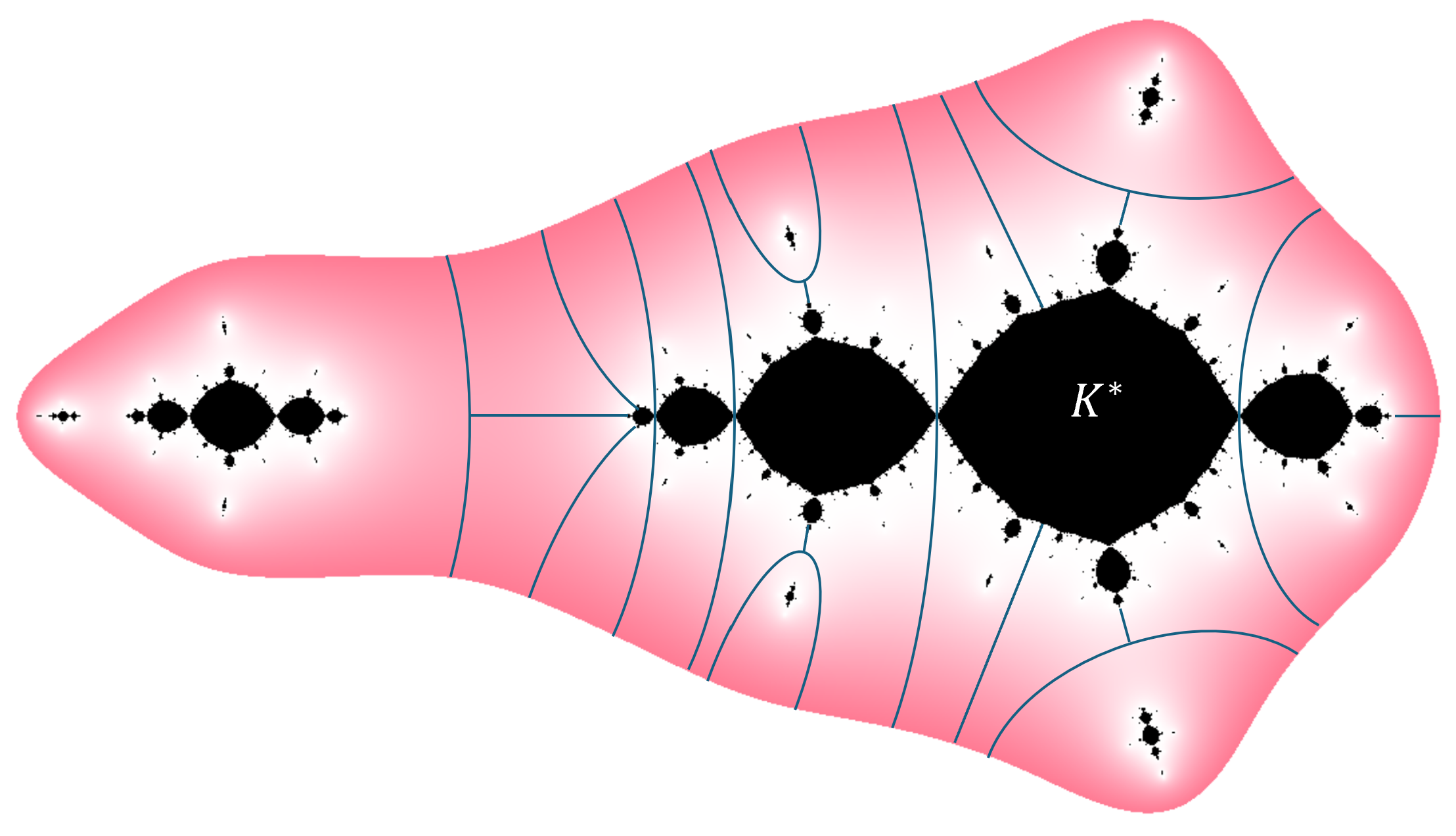}
  \caption{An equipotential disk for a cubic polynomial $f$ with a PL restriction $f:U^*\to f(U^*)$
   that is hybrid equivalent to $z\mapsto z^2-1$.
   (Parts of) some external intervals of $f$ are shown; all these are parts of
   broken rays to $K^*$.}\label{fig:rays}
\end{figure}

\begin{dfn}\label{d:curves}
For each $t$, let $E(t)$ be the equipotential of $f$ of level $t$;
 note that $E(t)$ may have a few connected components and be non-smooth.
Let $O(t)$ be the component of $E(t)$ separating $K^*$ from infinity, and
$O^*(t)\subset O(t)$ be the Jordan curve containing $K^*$ in its topological hull ($O^*(t)$
is well defined and may be a proper subset of $O(t)$).
\end{dfn}

Note that $f(E(t))=E(d\cdot t)$, as follows immediately from the properties of the Green function.
By definition, we also have that $f(O(t))=O(d\cdot t)$ and $f(O^*(t))=O^*(d\cdot t)$.
Let $t>0$ be a level of $g_f$ such that $O(t)$ contains singularities of $g_f$.
Compact disks bounded by $\ol{O(t)\sm O^*(t)}$ are called \emph{bubbles of level} $t$.

\begin{lem}
  \label{l:foam-cont}
  The topological hull of $O(t)$ is contractible.
It follows that bubbles are arranged in a tree-like fashion: any level $t$ bubble is connected to $O^*(t)$
 through a unique sequence of bubbles so that adjacent bubbles in the sequence have one
 boundary point in common.
\end{lem}

\begin{proof}
Pick $t'>t$ such that $O(t')$ is smooth, and the open annulus between $O(t)$ and $O(t')$
 contains no singularities of $g_f$.
Note that the topological hull $\th(O(t))$ is a deformation retract of the
 topological disk $\th(O(t'))$ bounded by the curve $O(t')$.
Indeed, a homotopy from the identity map of $\th(O(t'))$ to a retraction of $\th(O(t'))$ onto $\th(O(t))$
 can be made along external intervals.
Since the disk $\th(O(t'))$ is contractible, so is $\th(O(t))$.
The last claim of the lemma follows.
\end{proof}

Nontrivial trees of bubbles are possible when several singularities of $g_f$ appear on the same equipotential.
If $f$ has sufficiently many escaping critical points, then one can arrange
 for $O(t)$ to consist of many bubbles with an arbitrary tree combinatorics.
An alternative description of $O(t)$ is that of a pinched unit circle, where the pinching
 is performed along several chords that do not cross each other in the open unit disk.
Such collections of chords are known as \emph{finite geodesic laminations}.
Even though this language will not be used in the current paper,
 it may be worth noting that there is a $\si_d$-invariant geodesic lamination
 that encodes the pinching of $O(t)$, for all values of $t>0$ simultaneously.

\subsection{Invariant gaps}
We now turn from the analytic world to the combinatorial world.
\emph{Gaps} are defined as certain convex subsets of the closed disk
 $\ol\disk=\{z\in\C\mid |z|\le 1\}$.
They first appeared in the context of invariant geodesic laminations \cite{thu85}.
Under various guises, invariant gaps have been studied in \cite{GM93,IK12,PZ22}.

\begin{dfn}[Gaps]
For a compact set $X\subset \uc$ of more than 2 points and no interior, the convex hull $\ch(X)$ is called a \emph{gap}
(here and in what follows convexity is understood in terms of real plane geometry).
If $\ch(X)$ is denoted by $G$, the set $X$ is denoted by $G'$ and is called the \emph{base} of $G$.
Components of $\uc\sm X$  are called \emph{holes} of $X$ (or of $\ch(X)$).
Holes can be written as $(a;b)$, for some $a$, $b\in\uc$, the \emph{endpoints} of the hole,
 so that the motion from $a$ to $b$ within $(a;b)$ is in the positive (counterclockwise) direction.
For a hole $U$ of $X$, the edge of $\ch(X)$ that
 separates the interior of $\ch(X)$ from $U$ is said to be
\emph{associated} with $U$.
\end{dfn}

Also, recall that the \emph{$d$-tupling} map $\si_d$ of $\uc$ to itself
is written as $\si_d(x)=dx\ (\mod\, 1)$ in the additive notation (that is, viewing $x$ as an element of $\R/\Z$)
and as $\si_d(z)=z^d$ in the multiplicative notation (that is, viewing $z$ as a point in $\uc$).

\begin{dfn}[Invariant gaps]
Suppose that $G'\subset \uc$ is a compact subset such that $\si_d(G')=G'$ and,
for every hole $(\al;\be)$ of $G'$, either $\si_d(\al)=\si_d(\be)$, or the arc $(\si_d(\al); \si_d(\be))$
is also a hole of $G'$. Then $G=\ch(G')$ is said to be a \emph{(stand alone) invariant} gap.
\end{dfn}

An invariant gap $G$ is called \emph{finite} or \emph{infinite} depending on whether $G'$ is finite or infinite.
In this paper, we will mostly consider infinite (in fact, uncountable) gaps.
Define an equivalence relation $\sim_G$ as the minimal closed equivalence relation on $\uc$ such that
$x\sim_G y$ whenever $x$ and $y$ are in
 the same hole of $G$.
Every class of $\sim_G$ is therefore an at most countable connected union of holes, their endpoints and their limit points.
Alternatively, let $G'_p$ be the \emph{perfect part} of $G$, that is, the maximal perfect subset of $G'$.
Classes of $\sim_G$ can be described as the closures of the holes of $G'_p$. 
In particular, if $G'$ is countable then there is only one class of $\sim_G$ and the associated quotient space
 of the unit circle is a point.
On the other hand, if $G'$ is uncountable, then $G'_p$ is also uncountable, and the associated quotient space is a topological circle.
The following theorem is well-known.

\begin{thm}
  \label{t:desc-gap}
If $G'_p$ is non-degenerate, then there is an identification of $\uc/\sim_G$ and $\uc$ under which
the quotient map of $\uc$ to $\uc/\sim_G$ semi-conjugates
$\si_d|_{G'}$ to $\si_{d_*}$ for some $d_*<d$.
\end{thm}

The number $d_*(G)=d_*$ associated with $G$ as in Theorem \ref{t:desc-gap} is called the \emph{degree} of $G$
and is denoted by $\deg(G)$.
We will always assume in this paper that $d_*>1$.
Also, we fix an identification of $\uc/\sim_G$ with $\uc$ (as described in Theorem \ref{t:desc-gap})
 so that $\si_{d_*}\circ\psi_G=\psi_G\circ\si_d$ on $G'$, where $\psi_G:\uc\to\uc/\sim_G$
 is the quotient projection.

\begin{dfn}[Major holes]
\emph{Lengths} in $\uc$ refer to the standard Lebesgue \emph{probability} measure;
 in particular, the total length of $\uc$ is 1 rather than $2\pi$.
The length of a subset $E\subset \uc$ is denoted by $|E|$.
\emph{Major holes} (of an invariant gap $G$) are defined as holes of length at least $\frac{1}{d}$
(equivalently, as holes $(a;b)$ such that $\si_d$ is not injective on the closed arc $[a;b]$).
The edges of $G$ associated to major holes of $G$ are called \emph{major edges of $G$}, or simply \emph{majors of $G$}.
\end{dfn}

Let us now describe types of major holes. Extend $\si_d$ over the boundary $\partial G$
of $G$ linearly and still denote the extended map by $\si_d$.
The following notions can be introduced in a more general context of laminations
 but, for our purposes, it suffices to concentrate on invariant gaps.

\begin{dfn}[Types of major holes]
A major hole $(a;b)$ (of an invariant gap $G$) is said to be \emph{critical} if $\si_d(a)=\si_d(b)$, that is,
if the length of $(a;b)$ is an integer multiple of $\frac{1}{d}$. A chord of $\uc$ whose endpoints
have equal $\si_d$-images is said to be \emph{critical} (edges of $G$ associated with
critical holes of $G$ are said to be \emph{critical}, too).
A hole is of \emph{(pre)critical type} if its edge eventually maps to a critical edge,
 of \emph{periodic type} if its associated edge is periodic, and
 of \emph{(pre)periodic type} if its edge eventually maps to a periodic edge.
\end{dfn}

As an example, observe that a hole $T$ with fixed endpoints is a major hole of periodic type.
The fact that $G$ is an invariant gap implies that edges of $G$ map onto edges of $G$ except when
the hole is critical, in which case the associated edge maps to a point of $G'$.
Holes are mapped onto holes except when a hole is major
 (see, e.g., \cite[Lemma 2.28]{BOPT20}).
Note that a major hole of $G$ can be precritical without being critical.

\begin{dfn}[Gaps generated by critical holes]
\label{d:GA}
Consider a set $\A$ of pairwise disjoint open arcs of $\uc$ such that the length of each is $\frac kd$ for some integer $k>0$.
Then $\A$ defines the set $G'_\A$ of all points $x\in\uc$ such that $\si^n(x)\notin \bigcup\A$, for all $n\in\Z_{\ge 0}$.
Here and in what follows, $\bigcup\A$ stands for the union of all arcs in $\A$.
It is not hard to see that $G_\A:=\ch(G'_\A)$ is an invariant gap whose degree is $d\cdot |\uc\sm\bigcup\A|$.
\end{dfn}

Any invariant gap can be realized as $G_\A$ for some $\A$ as above, up to a countable set.
Moreover, one can arrange that each arc of $\A$ has length exactly $\frac 1d$.

\begin{thm}
\label{t:dyn-holes}
Let $G$ be a $\si_d$-invariant gap of degree $d_*$.
Every hole of $G$ is eventually mapped to a major hole.
Every hole of $G$ is either of (pre)critical type or of (pre)periodic type.
There is a collection $\A$ of $d-d_*$ disjoint open arcs of length $d^{-1}$ each such that
$G\subset G_\A$, and the set $G'_\A\sm G'$ is at most countable.
\end{thm}

\begin{proof}
The first two claims of Theorem \ref{t:dyn-holes} are well known
(see, e.g., \cite[Lemma 2.28]{BOPT20}). To prove the third claim define the family $\A$ of open arcs of length
$d^{-1}$ as follows.
Let $(a; b)$ be a major hole of $G$.
Insert in $I$ a few points $d^{-1}$ apart starting at $a$ and stopping
once the length of the remaining arc is less that $d^{-1}$.
If we do it in all major holes, then we obtain a family $\A$ of pairwise disjoint open arcs of length $d^{-1}$.

Consider a hole $I=(x;y)$ of $G$.
To describe the set $G_\A\cap I$, we iterate $\si_d$ on $I$.
On each step, we remove points of $I$ sent to arcs from $\A$.
Consider the first moment, say, $k$, when $\si_d^k(I)$ is a major hole $T$.
If $T$ is of length $d^{-1}$ then, by our construction, $T\in \A$. Hence there are no points of
$G_\A$ in $I$. Otherwise there are several concatenated arcs from $\A$ that start at the endpoint $\si_d^k(x)$
of $T$ and extend all the way to $\si_d^k(y)$ (if $T$ is critical) or to some point $z'\in T$ such that
$|(z';\si_d^k(y))|<d^{-1}$.

Their pullbacks in $I$ are concatenated arcs of length $d^{-k}\cdot d^{-1}$ with the
endpoints that belong to $G_\A$ such that if they are ordered in the counterclockwise direction
then the first one has an endpoint $x$.
If $T$ is critical, then the endpoints of the just described arcs are the only points of $G_\A$ inside $I$.
Otherwise there exists an arc
$(z;y)$ with $|(z;y)|<d^{-k-1}$ such that $\si^k_d(z;y)=(z'; \si_d^k(y))$ and, hence,
$\si^{k+1}_d(z;y)=(\si_d^{k+1}(x);\si_d^{k+1}(y))=Y$ is, again, a hole of $G$; after that the arguments can be repeated for $Y$.

The number of times (=values of $n$) when $\si_d^n(x;y)$ is a major hole of $G$
may be either finite or infinite.
In the former case, take all $n$'s where $\si_d^n(x;y)$ is a major hole, and take
the corresponding $\si_d^n$-pullbacks in $I$ of the arcs from $\A$.
The set $G_\A\cap I$ will consist of finitely many endpoints of these pullbacks.
On the other hand, in  the latter case, there are countably many subintervals of $I$
that are eventual pullbacks of arcs from $\A$; the countable set of
their endpoints coincides with $G_\A\cap I$.
Note that this countable set accumulates only to $y$.
\end{proof}

Though the set $\A$ with properties claimed in Theorem \ref{t:dyn-holes} is not unique,
for our purposes, any $\A$ with these properties will do.

\subsection{Fundamental nest}
\label{ss:fund}
In this section, we consider a monic centered degree $d$ polynomial $f$
 such that $K^*\subset K(f)$ is a degree $d_*>1$ invariant component.
Associated with it is a $\si_d$-invariant degree $d_*$ gap defined below, in Section \ref{ss:G}.
Recall that $U^*$ is an open Jordan disk such that $\d U^*=O^*(\eps)\subset O(\eps)$ for some $\eps_f=\eps>0$
which is the value of $g_f$ on $O(\eps)$,
there is at least one critical point of $f$ in $\d U^*$,
and all critical points of $f$ inside $U^*$ belong to $K^*$.
Set $t_n:=d^n\eps$, for any $n\in\Z$.
Let $A_n$ be the open annulus between $O^*(t_{n})$ and $O^*(t_{n+1})$.
These $A_n$ are nested annuli that, together with the curves $O^*(t_n)$ for all $n\in\Z$, fill $\C\sm K^*$
and converge, as $n\to -\infty$, to the boundary of $K^*$.
We will refer to the sequence of the annuli $A_n$ as the \emph{fundamental nest},
and to the annuli $A_n$ with $n\le 0$ as \emph{fundamental annuli}.
The following fact is standard and is widely used in polynomial-like dynamics.

\begin{lem}
\label{l:mapA_n}
For every $n>0$, the map $f:A_{-n}\to A_{-n+1}$ is a covering of degree $d_*$
(thus, $\mod(A_{-n})=d_*^{-1}\mod(A_{-n+1})$).
Also, $f:A_n\to A_{n+1}$ is a $d$-fold covering, for all sufficiently large positive $n$.
\end{lem}

A lower bound on $\mod(U^*\sm K^*)$ follows from the Gr\"{o}tzsch inequality:
$$
\mod(U^*\sm K^*)\ge \sum_{n> 0}\mod(A_{-n}).
$$
This, together with Lemma \ref{l:mapA_n}, yields the following corollary.

\begin{cor}
  \label{c:Gro}
One has $\mod(U^*\sm K^*)\ge \frac{1}{d_*-1}\mod(A_{0})$.
\end{cor}

Thus, to obtain an explicit lower bound on $\mod(U^*\sm K^*)$ it remains to estimate $\mod(A_{0})$ from below.

\begin{dfn}[Connectors and voids]
Let $m$ be the smallest integer such that $O(t)$ is smooth for all $t\ge t_{m+1}$.
Identify $O(t_{m+1})$ with $\uc$ by means of external arguments.
Define the compact subset $\Cc_n\subset \uc$ as the intersection of $O(t_{m+1})$
with the union of all external rays or broken rays that cross or hit the oval $O^*(t_n)$.
Components of $\Cc_n$ are called \emph{level $n$ connectors} and components of  $\uc\sm \Cc_n$ are called
\emph{level $\ge n$ voids} (of $(f,K^*)$).
If $V$ is a level $\ge n$ void but not a level $\ge n+1$ void, say that $V$ is a \emph{level $n$} void.
\end{dfn}

When studying connectors and voids, the following concept is useful.
For every external angle $\ta\in\R/\Z$ there is a path $R^*(\ta)=R^*_f(\ta)$ that consists of
 (parts of) external intervals and singularities of $f$
 and that extends either all the way to $K^*$ or to $O(t)$ with the smallest possible value of $t>0$.
In other words, one builds $R^*(\ta)$ starting from an initial segment of $R(\ta)$ and descending with respect to $g_f$;
 as $R^*(\ta)$ meets a singular equipotential component $O(t)$, one either continues it
 further to the bounded complementary component of $O(t)$ containing $K^*$
 or, if this is impossible, stops.
Note that $R^*(\ta)$ may contain one or several singularities of $g_f$;
 on the other hand, it may contain no singularities of $g_f$, lie in $R(\ta)$ and terminate on the boundary of some bubble.
Call $R^*(\ta)$ the \emph{ray towards $K^*$} (as opposed to a ray \emph{to} $K^*$).
Note that $\ta\in\Cc_n$ if and only if $R^*(\ta)$ reaches $O^*(t_n)$.

\begin{lem}
  \label{l:Ccn}
One has $\si_d(\Cc_n)=\Cc_{n+1}$.
Every level $n$ connector maps into, not necessarily onto, a level $n+1$ connector.
\end{lem}

\begin{proof}
Suppose that $\ta\in \Cc_n$, which means that $R^*(\ta)$ reaches $O^*(t_n)$.
Then, clearly, $R^*(\si_d(\ta))=f(R^*(\ta))$ reaches $f(O^*(t_n))=O^*(t_{n+1})$.
Assume now that $R^*(\ta)$ reaches $O^*(t_{n+1})$, say, at a point $z\in O^*(t_{n+1})$.
Choose any $f$-preimage $z'$ of $z$ that lies in $O^*(t_n)$, and
 let $R'$ be the pullback of $R^*(\ta)$ containing $z'$.
It follows that $R'=R^*(\ta')$ for some $\si_d$-preimage $\ta'$ of $\ta$,
 which means that $\ta'\in\Cc_n$ and $\ta=\si_d(\ta')\in\si_d(\Cc_n)$.
The last claim of the lemma is immediate.
\end{proof}

We can now describe a geometric meaning of a level $n$ void $(\al;\be)$
 in terms of the dynamical plane of $f$.
No ray towards $K^*$ between $R^*(\al)$ and $R^*(\be)$ reaches the inner boundary $O^*(t_n)$ of $A_n$.
It follows that $R^*(\al)$ and $R^*(\be)$ merge at a point $z$ with $g_f(z)\ge t_n$.
On the other hand, since $(\al;\be)$ is a level $n$ void, $g_f(z)<t_{n+1}$.
Also, no ray towards $K^*$ whose external argument is in $(\al;\be)$ merges with $R^*(\al)$ or $R^*(\be)$.
Thus, a level $n$ void $(\al;\be)$ corresponds to a pair of rays $R^*(\al)$, $R^*(\be)$ towards $K^*$
 that merge in $A_n$ and are such that no $R^*(\ga)$ with $\ga\in(\al;\be)$ merges $R^*(\al)\cup R^*(\be)$.

\begin{lem}
  \label{l:voids}
Let $(\al;\be)$ be a level $n$ void of $(f,K^*)$.
Either $\si_d(\al)=\si_d(\be)$, or $(\si_d(\al);\si_d(\be))$ is a level $>n$ void.
Voids of $(f,K^*)$ are either disjoint or nested.
\end{lem}

\begin{proof}
Since $\al\in\R/\Z$, we may write $\si_d(\al)$ and $d\al$ interchangeably.
Let $z$ be the point where $R^*(\al)$ and $R^*(\be)$ merge.
Assuming that $\si_d(\al)\ne\si_d(\be)$, consider the $f$-images $R^*(d\al)$ and $R^*(d\be)$
 of $R^*(\al)$ and $R^*(\be)$.
These images merge at a point $z'$ with $g_f(z')\ge dg_f(z)$ (it is possible that $g_f(z')>dg_f(z)$ if $z$ is a critical point of $f$)
 thus shielding $R^*(\ga')$ from $K^*$, for all $\ga'\in (d\al;d\be)$.
It follows that $(d\al);d\be)$ is a level $>n$ void, as claimed.

Finally, let $(\al;\be)$ and $(\ga;\delta)$ be two voids of $(f,K^*)$ that are neither
 disjoint nor nested; we may assume that $\al<\ga<\be<\delta$ in the circular order of $\R/\Z$.
By the above, $R^*(\al)$ and $R^*(\be)$ merge
 and therefore shield $R^*(\ga)$ from $R^*(\delta)$, unless all four rays merge.
Since $\ga\in (\al;\be)$, this contradicts $(\al;\be)$ being a void.
\end{proof}

\subsection{The gap}
\label{ss:G}
We can now associate a certain gap $G$ with $(f,K^*)$.

\begin{lem}
\label{l:G}
Let $G'$ be the intersection of $\Cc_n$ for all $n\in\Z$. Set $G:=\ch(G')$.
Then $G$ is a degree $d_*$ invariant gap.
\end{lem}

\begin{proof}
Since $\si_d(\Cc_n)=\Cc_{n+1}$, by Lemma \ref{l:Ccn},
 and since the image of the intersection lies in the intersection of the images, we have $\si_d(G')\subset G'$.
Now, take $\ta\in G'$.
As $f:U^*\to f(U^*)$ is proper and $K^*$ is fully invariant under $f|_{U^*}$,
 there is an $f$-pullback of $R^*(\ta)$ that also accumulates in $K^*$.
It follows that $\ta=\si_d(\ta')$ for some $\ta'\in G'$, and we see that $\si_d(G')=G'$.
What remains to show is that, given a hole $(\al;\be)$ of $G$, either $\si_d(\al)=\si_d(\be)$,
 or $(\si_d(\al);\si_d(\be))$ is a hole of $G$.

Assume that $\si_d(\al)\ne\si_d(\be)$.
By Lemma \ref{l:voids}, voids of $(f,K^*)$ are either disjoint or nested.
Hence, $(\al;\be)$ can be either a void, or the union of an increasing sequence of voids.
In the former case, the desired property follows from Lemma \ref{l:voids}.
Consider the latter case: there are sequences $\al_n\searrow\al$ and $\be_n\nearrow\be$ (as $n\to +\infty$)
 such that $(\al_n;\be_n)$ is a level $-n$ void, for all $n>0$.
By Lemma \ref{l:voids}, the arcs $(d\al_n;d\be_n)$ are also finite level voids.
Hence,
$$
(d\al;d\be)=\bigcup_{n>0} (d\al_n;d\be_n)
$$
 is contained in a hole of $G$.
On the other hand, since $R^*(\al)$ and $R^*(\be)$ extend all the way to $K^*$,
 so do the images $R^*(d\al)$ and $R^*(d\be)$.
It follows that $(d\al;d\be)$ is a hole of $G$.

By definition, $G$ is a $\si_d$-invariant gap.
From the fact that the PL map $f:U^*\to f(U^*)$ has degree $d_*$, it now easily follows that
 the gap $G$ also has degree $d_*$.
\end{proof}

Lemma \ref{l:G} justifies the following definition.

\begin{dfn}[Renormalization gap of $(f,K^*)$]
\label{d:renormG}
The $\si_d$-invariant gap $G$ associated with $(f,K^*)$ as in Lemma \ref{l:G}
 is called the \emph{renormalization gap of $(f,K^*)$}.
\end{dfn}

A void $(\al;\be)$ gives rise to an associated \emph{cut}, the simple curve $\Gamma_{\al\be}$ from infinity
 to infinity contained in the union $R^*(\al)\cup R^*(\be)$.
The complement of $\Gamma_{\al\be}$ in $\C$ consists of two components;
 the component $W_{\al\be}$ not containing $K^*$ is called the \emph{wedge} of $(\al;\be)$.
Say that the void $(\al;\be)$ is \emph{major} if $f$ is not injective on the closure of $W_{\al\be}$.
The next proposition is immediate from the definitions.

\begin{prop}
\label{p:mvoid}
If a void is not major, then its $\si_d$-image is a void of a higher level.
Every void is eventually mapped to a major void.
\end{prop}

Note that the notion of renormalization gap extends to the more general case when $K^*$ is the
 filled Julia set of any PL restriction of $f$; in this more general setup, $K^*$ is still assumed to be connected,
 but it is not necessarily a component of $K(f)$.

\section{Metric theory of invariant gaps}
\label{s:metr}

In this section, we establish certain estimates of the number and the length of the
 level connectors associated with $(f,K^*)$.
These yield upper and lower bounds on the Hausdorff measure of $G'$ in its
 Hausdorff dimension, as well as on the Hausdorff measures of segments of $G'$.

\subsection{Connectors}
\label{ss:holconn}
Recall that a connector of level $n$ is a component of the set $\Cc_n$.
Also, recall that $m$ is the smallest non-negative integer with the property $\Cc_{m+1}=\uc$.
Previously, we identified $\uc$ with $O(t_{m+1})$.

\begin{dfn}\label{d:bn}
The number of components of $\Cc_{n}$ is denoted $b_{n}$, and the
maximal length of a component of $\Cc_{n}$, for $n\ge 0$, is denoted $\la_{n}$.
\end{dfn}

Under the map $\si_d$, a generic point of $\Cc_{n+1}$ has $d_n$ preimages in $\Cc_n$, where $d_*\le d_n\le d$.
For $n\ge 0$, one has $d_{-n}=d_*$, while for $n>m$, one has $d_n=d$.

In what follows, we will use Lemma \ref{l:linear} stated without proof.

\begin{lem}\label{l:linear}
Let $L:\R\to\R$ be a real affine function with fixed point $a$ and multiplier $\ga>1$ at $a$.
 Suppose that a sequence
$u_n\in\R$ satisfies $a<u_{n+1}\le L(u_n)$. Then $u_k\le L^k(u_0)=a+\ga^k(u_0-a)$.
\end{lem}

We apply Lemma \ref{l:linear} to obtain inequalities on $b_n$.

\begin{lem}
  \label{l:bn0}
If $m>0$ and $0\le n\le m$, then
$
b_n< (d-d_*)(d-1)^{m-n}.
$
\end{lem}

\begin{proof}
Note that $\si_d:\Cc_n\to \Cc_{n+1}$ can be represented as a
 composition of two maps: firstly, we collapse the closure of every level $\ge n$ void of length $\frac id$,
 where $i$ is an integer; secondly, we apply a certain degree $d_n$ covering.
The number of components in $\Cc_n$ is therefore at most $d_n$ times $b_{n+1}$ plus the number
 of collapsed voids; the latter is bounded above by $d-d_n$.
Since $2\le d_n\le d-1$ for $0\le n\le m$, then
for $n\in [0,m]$ one has
$$
b_n\le d-d_n + d_nb_{n+1}\le d-2+(d-1)b_{n+1}.
$$
Set $L(x)=(d-1)x+d-2$; since $-1$ is the $L$-fixed point and $b_m\le d-d_*-1$, then, by
Lemma \ref{l:linear},
$$
b_n\le -1+[d-d_*-1-(-1)](d-1)^{m-n}<(d-d_*)(d-1)^{m-n},
$$
as claimed.
\end{proof}

Set
\begin{equation}\label{e:c0}
c_0:=\frac{d-d_*}{d_*-1}.
\end{equation}

\begin{lem}
  \label{l:bn}
For $n>0$, one has $b_{-n}<(b_0+c_0)d_*^n$.
If $m=0$, then $b_0\le d-d_*$ and $b_{-n}<(b_0+c_0)d_*^n=c_0\,d_*^{n+1}$.
\end{lem}

\begin{proof}
The proof is similar to the proof of Lemma \ref{l:bn0}.
First we establish that $b_{-n}\le d-d_*+d_*b_{-n+1}$.
Here, the additive term $d-d_*$
  is an upper bound for the number of components in $\Ic_{-n}$.
Then we define $L(x)=d-d_*+d_*x$.
The $L$-fixed point is $\frac{d-d_*}{1-d_*}=-c_0$. Moreover, $-c_0<b_{-n}\le L(b_{-n+1})$.
Hence, by Lemma \ref{l:linear}, $b_{-n}\le -c_0+d^n_*(b_0+c_0)$. The rest of the lemma follows from the fact that, if $m=0$, then
$b_0\le d-d_*$ and, hence, $b_{-n}<(b_0+c_0)d_*^n\le (d-d_*+c_0)d_*^n=c_0\,d_*^{n+1}$.
\end{proof}

Let $\la_n$ be the maximal length of a component of $\Cc_n$ for some $n\le m$.

\begin{lem}
  \label{l:lan}
For $n\le m$, one has $\la_n\le (d-1)d^{n-m-1}$.
In particular, $\la_0\le (d-1)d^{-m-1}$.
Also, $\la_{-n}\le \la_0d^{-n}$ whenever $n\ge 0$.
If $m=0$, then $\la_0\le d_*d^{-1}$.
\end{lem}

\begin{proof}
Evidently, $\la_m\le d_m d^{-1}\le (d-1)d^{-1}$.
It is easy to see that $\la_n\le d^{-1}\la_{n+1}$ for $n<m$.
Induction by $n$ from $n=m$ downwards
using the above inequality completes the proof.
When $m=0$, the upper bound on $\la_0$ is at most
 the total length of $\Cc_0$, and the latter is bounded above by $d_*d^{-1}$
 (since $\si_d$ is generically $d_*$-to-one, and it expands the lengths by the factor of $d$).
\end{proof}

Finally, we need to estimate the total length $\mu_n$ of $\Cc_n$.

\begin{lem}\label{l:mun}
If $0<n<m$, then $\mu_n=d_nd^{-1}\mu_{n+1}$ and, hence, $$d_*d^{-1}\mu_{n+1}\le \mu_n\le (d-1)d^{-1}\mu_{n+1}.$$
Thus, if  $0\le n\le m$ then
$$
d_*^{m+1-n}d^{n-m-1}\le \mu_n\le (d-1)^{m+1-n}d^{n-m-1},
$$
and, in particular,
$$
d_*^{m+1} d^{-m-1}\le \mu_0\le (d-1)^{m+1}d^{-m-1}.
$$
Also, if $k>0$ then
$
\mu_{-k}=d^{-k}d_*^k\mu_0.
$
\end{lem}

\begin{proof}
Using that $\si_d:\Cc_n\to\Cc_{n+1}$ is generically $d_n$-to-one for $0\le n\le m$ and that
$\si_d$ multiplies all lengths by $d$, we obtain that $\mu_n=d_nd^{-1}\mu_{n+1}$. On the other hand,
$d_*\le d_n\le d-1$. Hence,
$$d_*d^{-1}\mu_{n+1}\le \mu_n\le (d-1)d^{-1}\mu_{n+1}.$$
Applying these
inequalities $m+1-n$ times and using that $\mu_{m+1}=1$, we see that
$$
d_*^{m+1-n}d^{n-m-1}\le \mu_n\le (d-1)^{m+1-n}d^{n-m-1}.
$$
as claimed. The rest of the lemma immediately follows.
\end{proof}

\subsection{H\"{o}lder continuity and the Hausdorff dimension}
\label{s:Hol}
Section \ref{s:Hol} studies an abstract $\si_d$-invariant gap $G$ of degree $d_*$,
 without any regard to complex polynomials.
Recall that $\delta_*:=\log_d d_*$ so that $d_*=d^{\delta_*}$.
Below, we establish H\"{o}lder continuity of the map $\psi:=\psi_G$,
show that the Hausdorff dimension of $G'$ equals $\delta_*$, and estimate
the $\delta_*$-Hausdorff measure of $G'$.

Recall the concepts of the Hausdorff measure and
dimension (in that we follow \cite{Rog98,Fal03}).
Given a metric space $Z$ and
numbers $s\ge 0$ and $\e>0$, define
\def\Hmes{\mathrm{Hmes}}
$$
\Hmes^s_\e(Z)=\inf\left\{
\sum_{i=1}^{\infty} \diam(U_i)^s\mid \{U_i\} \hbox{
covers $Z$ with $\diam(U_i)\le \e$}
\right\}.
$$
Here, $\diam(U)$ means the diameter of a set $U$.
As $\e\searrow 0$ with fixed $Z$ and $s$, the value of $\Hmes^s_\e(Z)\in [0;+\infty]$ non-strictly increases;
 the limit $\Hmes^s(Z)$ of this value as $\e\to 0$ is called the \emph{$s$-Hausdorff measure of $Z$}.
It is either a non-negative real number or infinity.
Define the \emph{Hausdorff dimension} $\Hdim(Z)$ of $Z$ as
$$
\sup\{s\mid \Hmes^s(Z)=\infty\}=\inf\{s\mid \Hmes^s(Z)=0\}.
$$
The Hausdorff dimension of a metric space is well defined, and if $\Hmes^s(Z)$ is a positive number, then $\Hdim(Z)=s$.
Proposition \ref{p:Falc} below is a special case of \cite[Theorem 2.9]{Rog98}, see also Proposition 2.2 of \cite{Fal03}
 for the Euclidean case.

\begin{prop}
\label{p:Falc}
Let metric spaces $(Z,\rho_Z)$, $(W,\rho_W)$,
 and a map $F:Z\to W$ be such that, for some $c$, $\e$, $\alpha>0$,
 if $\rho_Z(x,y)<\e,$ then $\rho_W(F(x),F(y))\le c\cdot \rho_Z(x,y)^\alpha$.
Then
$\Hmes^s(Z)\ge c^{-s/\al}\Hmes^{s/\al}(F(Z))$ for all $s\ge 0$.
In particular,
$\Hmes^\al(Z)\ge c^{-1}\cdot \Hmes^1(F(Z))$.
\end{prop}

In \cite{Fal03}, the proof is given for Euclidean spaces, but the argument
 extends verbatim to arbitrary metrics.

Let us now return to the study of the map $\psi:=\psi_G$ and the Hausdorff dimension of $G'$
 (recall that $\psi$ collapses all holes of $G$ to points, and semi-conjugates $\si_d|_{G'}$ with $\si_{d_*}$).
By Theorem \ref{t:dyn-holes}, there is a collection $\A$ of $d-d_*$ disjoint open arcs in
$\uc$, each of length $d^{-1}$, such that $G\subset G_\A$, and the set $G'_\A\sm G'$ is at most countable.
Therefore, without loss of generality, we may assume from now on that $G=G_\A$. It is easy to see that
$\psi_G=\psi_{G_\A}=\psi_{G_p}$, where $G_p$ is the perfect part of $G$.
Also, $\A$ coincides with the set of all major holes of $G$
 (for our specific choice of $G=G_\A$, all major holes are critical).

\begin{dfn}[Connectors in a combinatorial setting]
Previously, we defined the sets $\Cc_n$ for a complex degree $d$ polynomial $f$ and an
 invariant degree $d_*$ component $K^*$ of $K(f)$.
Now, we have an abstract invariant gap $G$ generated by a set $\A$ of arcs.
In this setting, we can also define the sets $\Cc_n$.
Namely, $\Cc_{-n}$, for $n\ge 0$, consists of all $x\in\uc$ such that $\si_d^k(x)\notin\bigcup\A$ for $k=0$, $\dots$, $n$.
\end{dfn}

Components of $\Cc_{-n}$ are called level $-n$ connectors, as in the polynomial setting.
Observe that all estimates of Section \ref{ss:holconn} are applicable to $\Cc_{-n}$, with $m=0$.

\begin{lem}\label{l:psi-distance}
Let $I$ be a connector of level $-n<0$. Then for every $j\le n$ we have $\psi\circ \si_d^j=\si^j_{d_*}\circ \psi$ on $I$.
Moreover, $|\psi(I)|\le d_*^{-n}$.
\end{lem}

\begin{proof}
The first claim
 uses the fact that, except for major holes, the $\si_d$-image of a hole is a hole.
Note that $\psi\circ\si_d^j=\si^j_{d_*}\circ\psi$ on $I\cap G'$, by the above mentioned
 property that $\psi$ semi-conjugates $\si_{d}$ with $\si_{d_*}$ on $G'$.
Also, all components of $I\sm G'$ are collapsed to points by both $\psi\circ\si_d^j$ and $\si_{d_*}^j\circ\psi$.
To prove the second claim, notice that $\si_d^n(I)$ lies in a level 0 connector $T$. By the above,
$\si_{d_*}^n(\psi(I))=\psi(\si_d^n(I))$ which implies that $\si_{d_*}^n|_{\psi(I)}$ is one-to-one (except possibly
for the endpoints). Therefore, $|\psi(I)|\le d_*^{-n}$.
\end{proof}

Complementary components of $\Cc_n$ are open arcs $U$ in $\uc$
with the property that $\si^k_d(U)\in\A$ for some $k\ge 0$ with $k\le n$ while $U$, $\dots$, $\si^{k-1}_d(U)$ are disjoint from $\bigcup\A$.
Such $U$ is called a \emph{level $-k$ hole}. If $n\ge 0$ then
the length of a level $-n\le 0$ hole is $d^{-n-1}$.
We may also refer to level $-k$ holes as \emph{level $-k$ voids}, since, in our current setting,
 voids (complementary components of $\Cc_{-k}$) coincide with holes of $G$.

We want to apply Proposition \ref{p:Falc} to the set $G'$.
To this end, we define a distance function $\rho$ on $\uc$ as follows.

\begin{dfn}\label{d:shortest}
For any circle arc $T$ let $|T|$ be the length of $T$.
For $x, y\in \uc$, let $Q_{xy}$ be the shortest circle arc between $x$ and $y$, and let $\rho(x, y)=|Q_{xy}|$.
\end{dfn}

Hausdorff dimension and Hausdorff measure in $\uc$ are understood in the sense of this metric.

\begin{lem}
  \label{l:psi-Hol}
If $x,y\in\uc$, then
$\rho(\psi(x),\psi(y))\le
 d_*^2\cdot \rho(x,y)^{\delta_*},$
and so $\psi$ is $\delta_*$-H\"older continuous.
\end{lem}

\begin{proof}
If
 $Q_{x,y}$ contains a level $0$ hole, then $\rho(x,y)\ge d^{-1}$ and
$$
d_*^2\cdot \rho(x,y)^{\delta_*}\ge d_*^2\cdot d^{-\de_*}=d_*>\rho(\psi(x),\psi(y))
$$
Hence we may assume that $n>0$
is such that $-n$ is the greatest level of a hole $U\subset  Q_{x,y}$.
Then $|U|=d^{-n-1}\le \rho(x,y)$.
On the other hand, 
 there is only one level $-n+1$ connector $T$ such that $Q_{x,y}\cap T\ne\0$;
 indeed, otherwise, if there were two or more such connectors, then $Q_{x,y}$ would contain a level $\ge -n+1$ hole
 (a contradiction with our assumption).
By Lemma \ref{l:psi-distance}, the $\rho$-diameter of $\psi(T)\supset \psi(Q_{x,y})$
 is at most $d_*^{-n+1}$, hence
 $\rho(\psi(x),\psi(y))\le d_*^{-n+1}$.
Thus,
$$
\rho(\psi(x),\psi(y))\le d_*^{-n+1}=d_*^2\cdot d_*^{-n-1}=
d_*^2\cdot d^{\delta_*(-n-1)}
\le d_*^2\rho(x,y)^{\delta_*},
$$
as claimed.
\end{proof}

One can compute $\Hdim(G')$ using Proposition \ref{p:Falc}.

\begin{thm}
\label{t:HdimG}
The following formulas hold:

\begin{enumerate}
\item $\Hdim(G')=\delta_*=\log_d d_*<1$;

\item $\Hmes^{\delta_*}(G'\cap [a;b])\ge d_*^{-2} |[\psi(a);\psi(b)]|$
for any arc $[a;b]\subset\uc$ and, in particular, $\Hmes^{\delta_*}(G')\ge d_*^{-2}$.

\item $\Hmes^{\delta_*}(G')\le c_0\,d_*^{\delta_*}$.

\end{enumerate}

\end{thm}

Recall that $c_0=\frac{d-d_*}{d-1}$ (equation (\ref{e:c0})).

\begin{proof}
By Lemma \ref{l:psi-Hol}, the function $\psi: {G'}\to\uc$ satisfies the assumptions of Proposition \ref{p:Falc}
 with $c=d_*^2$ and $\alpha=\delta_*$.
Take $s=\delta_*$.
By Proposition \ref{p:Falc}, $\Hmes^1(\uc)=1\le d_*^2\,\Hmes^{\delta_*}({G'})$,
which implies that $\Hmes^{\delta_*}({G'})\ge d_*^{-2}$.
The inequality for $[a;b]$ is similar.
By definition, this yields that $\Hdim({G'})\ge \delta_*$.

To estimate $\Hmes^{\delta_*}({G'})$ from above, we use the covering
 of $G'$ by all level $-n$ connectors.
By Lemma \ref{l:bn}, there are at most $c_0\,d_*^{n+1}$
 connectors of level $-n$ each of which is no longer than $d^{-n}\la_0$, by Lemma \ref{l:lan}.
Hence the sum of the lengths of these connectors raised to the power $\delta_*$ is at most
$(d^{-n}\la_0)^{\delta_*} c_0\,d_*^{n+1}=\lambda_0^{\delta_*} c_0\, d_*$.
It remains to use that $\la_0\le d_*d^{-1}$ (Lemma \ref{l:lan}) to obtain
inequality (3). We conclude that $\Hdim({G'})=\delta_*=\log_d{d_*}<1$.
\end{proof}

Theorem \ref{t:HdimG} gives all we need to deduce Theorem \ref{t:main-comb}.

\begin{proof}[Proof of Theorem \ref{t:main-comb}]
In fact, Theorem \ref{t:HdimG} includes a reformulation of Theorem \ref{t:main-comb}:
 the Hausdorff dimension of $G'$ is found in (1); the lower bound on the Hausdorff measure is contained in (2),
 while the upper bound is (3).
\end{proof}

\subsection{Measures}
\label{ss:meas}
Our setup in Section \ref{ss:meas} is as follows:
 $G$ is the renormalization gap of $(f,K^*)$,
 and $\Cc_i$ is the union of all level $i$ connectors of $(f,K^*)$.
We cannot assume anymore that all escaping critical points of $f$ have the same escape rate.
Recall that $\psi_G=\psi$ is a factor map of $\uc$ to $\uc/\sim_G$ semi-conjugating $\si_d|_{G'}$ and $\si_{d_*}$
with $d_*>1$.
Without loss of generality, we may assume that $0\in G'$ and $\psi(0)=0$.

\begin{dfn}
\label{d:psiex}
Let $B=\bigcup_{j=1}^k I_j\ne \uc$ be the union of disjoint non-degenerate closed arcs.
Define
$\psi_B:\uc\to\uc$
 as the unique continuous function such that $\psi_B(0)=0$, $\psi'(x)=1/|B|$ on the interior of $B$ and $\psi'(x)=0$ on
$\uc \sm B$. Let $\psi_i$ denote the map $\psi_{\Cc_i}:\uc\to \uc$.
\end{dfn}

It is easy to see that there is indeed a unique function $\psi_B$ with the listed properties,
 for any given $B$;
 in particular, $\psi_i$ is well defined.

\begin{dfn}\label{d:l_n}
Define $\lf_i$ as the measure on $\uc$ which is the pullback of the Lebesgue measure on $\uc$ under $\psi_i$
(by definition of a pullback measure, this means that $\lf_i([\al;\be])=|\psi_i([\al,\be])|$ for all $\al$, $\be\in\uc$).
Also, set $\lf_\infty$ to denote the standard Lebesgue probability measure on $\uc$.
Thus, $\lf_\infty(E)$, where $E$ is an arc of $\uc$, is just an alternative notation for $|E|$.
By the assumptions, $\lf_\infty=\lf_n$ for all $n>m:=\max(\chi)$.
\end{dfn}

Consider all major voids of level $0$.
In each such void, choose and fix an arc $U$ of length $\frac kd$,
 for the largest possible integer value of $k$.
Form the set $\A$ of all such arcs $U$; this set will play a certain role
 in the upcoming constructions.
Note that the gaps $G$ and $G_\A$ have the same perfect part,
 since, evidently, $G\subset G_\A$, and $G$, $G_\A$ have the same degree.

Let $\si^\A$ be the self-map of the circle that equals $\si_d$ on the complement of $\bigcup\A$
and collapses the closures of all arcs from $\A$ to points.
Then $\psi_{-n+1}\circ\si^\A=\si_{d_*}\circ\psi_{-n}$ for all $n>0$.
Indeed, the maps on both sides collapse the closures of all components of $\uc\sm\Cc_{-n}$ to points,
and otherwise multiply all lengths by a uniform factor.
Now, equality $\psi_{-n+1}\circ\si^\A=\si_{d_*}\circ\psi_{-n}$
 follows from the uniqueness of the map $\psi_n$ based upon
Definition \ref{d:psiex}.
Hence, the following diagram is commutative:
\begin{equation}\label{e:dia1}
\xymatrix{
\dots \ar[r]^{\si^\A}    & \uc\ar[r]^{\si^\A} \ar[d]_{\psi_{-n}} & \uc \ar[d]_{\psi_{-n+1}}\ar[r]^{\si^\A}  & \dots \ar[r]^{\si^\A}
 & \uc\ar[d]^{\psi_{-1}}\ar[r]^{\si^\A} & \uc\ar[d]^{\psi_0} \\
\dots \ar[r]_{\si_{d_*}}    & \uc\ar[r]_{\si_{d_*}}                   & \uc \ar[r]_{\si_{d_*}}
& \dots \ar[r]_{\si_{d_*}}       & \uc   \ar[r]_{\si_{d_*}}           & \uc}
\end{equation}
Observe that the derivative $\psi'_{-n}$ of $\psi_{-n}$ is constant outside of the flat spots of $\psi_{-n}$. By
 diagram \eqref{e:dia1} and the chain rule, $\psi'_{-n+1}\cdot d=d_*\cdot \psi'_{-n}$.
Hence $\psi'_{-n}/\psi'_{-n+1}=d/d_*$ with derivatives taken at the appropriate points.
If $Z$ is a measurable set \emph{contained in the support of} $\lf_{-n}$, then, by definition,
 $\lf_{-n}(Z)=|\psi_{-n}(Z)|=\psi'_{-n}\cdot |Z|$ with $\psi'_{-n}$ being the slope of all non-horizontal pieces of the graph.
On the other hand, $\lf_{-n+1}(Z)=|\psi_{-n+1}(Z)|=\psi'_{-n+1}\cdot |Z|$.
Hence $\lf_{-n}(Z)=\lf_{-n+1}(Z)\cdot \psi'_{-n}/\psi'_{-n+1}=\lf_{-n+1}(Z)\cdot d/d_*$.

\begin{lem}
  \label{l:lk-scale}
For any integer $n\ge 0$ and any measurable subset $Z\subset \Cc_{-n}$
one has $\lf_{-n}(Z)=(\frac {d}{d_*})^n \lf_0(Z)$.
Moreover, if $\si_d^{n}|_Z$ is injective on the complement of finitely many points in $Z$,
then $d^{n}\lf_0(Z)=\lf_0(\si_d^{n}(Z))$ and $\lf_{-n}(Z)={d_*^{-n}} \lf_0(\si_d^{n}(Z))$.
\end{lem}

\begin{proof}
The first claim follows from the fact that $\lf_{-n}(Z)=\lf_{-n+1}(Z)\cdot d/d_*$ for every $n\ge 0$. The second claim is immediate.
Hence, if $\si_d^n|_Z$ is injective outside of a finite subset, then
 ${d_*^{-n}} \lf_0(\si_d^{n}(Z))={d_*^{-n}} d^{n}\lf_0(Z)=\lf_{-n}(Z)$ (the last equality follows from the first claim of the lemma).
\end{proof}

We can now establish the uniform convergence of the sequence $(\psi_{-n})$ as $n\to \infty$.

\begin{lem}
\label{l:psi-n}
The sequence of maps $\psi_{-n}:\uc\to\uc$ converges uniformly as $n\to \infty$.
Moreover, the uniform distance between $\psi_{-n}$ and $\psi_{-n+1}$ is at most $d_*^{-n}$.
\end{lem}

\begin{proof}
By the definitions and the commutativity of 
 diagram (\ref{e:dia1}),
 if $x\in G'$ is such that $\si_d^n(x)=0$
for some $n\ge 0$ then $\psi_{-m}(x)=\psi_{-n}(x)$ for any $m\ge n$.
The lemma now follows from the monotonicity of $\psi_{-n}$ and the fact that
the distance between two consecutive preimages of $0$ under $\si_{d_*}^k$ is $d_*^{-k}$.
\end{proof}

Recall that the map $\psi$ collapses the closures of all holes of $G$ to points,
fixes 0, and semi-conjugates $\si^\A$ with $\si_{d^*}$.

\begin{cor}
  \label{c:psilim}
The sequence $\psi_{-n}$ converges to $\psi$ as $n\to \infty$.
\end{cor}

\begin{proof}
By Lemma \ref{l:psi-n}, the uniform limit $\psi_{-\infty}$ of $\psi_{-n}$ as $n\to \infty$ exists and takes $0$ to $0$;
by definition, it collapses the closure of every hole of $G$.
Passing to the limit as $n\to \infty$ in $\psi_{-n+1}\circ\si^\A=\si_{d_*}\circ\psi_{-n}$ yields that $\psi_{-\infty}$
semi-conjugates $\si^\A$ with $\si_{d_*}$.
Since $\psi_{-\infty}$ satisfies all the defining properties of $\psi$, then $\psi_{-\infty}=\psi$.
\end{proof}

Recall that $\delta_*:=\log_d d_*$, and
 $d_*^{-2}\le \Hmes^{\delta_*}(G')\le c_0\,d_*^{\delta_*}$, by Theorem \ref{t:HdimG}.

\begin{dfn}\label{d:hf}
Denote by $\hf$ the Hausdorff $\delta_*$-measure on $G'$ extended by zero to all of $\uc$
 and normalized so that $\hf(\uc)=1$.
Thus, $\hf[\al;\be]=\Hmes^{\delta_*}([\al;\be]\cap G')/\Hmes^{\delta_*}(G')$ for every arc $[\al;\be]\subset \uc$.
Let $\lf_{-\infty}$ be the limit of the measures $\lf_{-n}$ as $n\to \infty$
(by Corollary \ref{c:psilim}, the limit measure $\lf_{-\infty}$ exists and equals the $\psi$-pullback of the Lebesgue measure
so that $\lf_{-\infty}(\al;\be)=|\psi(\al;\be)|$ for every arc $(\al;\be)\subset\uc$).
\end{dfn}

Both $\lf_{-\infty}$ and $\hf$ are probability measures on $\uc$ supported on $G'_p$
and invariant under $\si_d$ such that the $\psi$-pushforwards
of both measures are $\si_{d_*}$-invariant measures on $\uc$.
It is well known that an invariant absolutely continuous $\si_{d_*}$-invariant probability measure on $\uc$
 coincides with the Lebesgue measure.
Below, we adapt the proof of this fact to the $\psi$-pushforward of $\hf$,
which satisfies a stronger invariance property.
Observe that the pushforward $\psi_*\lf_{-\infty}$ of $\lf_{-\infty}$ under $\psi$ is the Lebesgue measure $\lf_\infty$.
A strong invariance property of $\lf_{-\infty}$ and $\hf$ is stated below.

\begin{lem}
  \label{l:hf-inv}
If $\si_d$ is one-to-one on the set $(\al;\be)\cap G'$ where $(\al;\be)\subset\uc$ is a circle arc,
 then $\hf(\al;\be)=\frac 1{d_*}\hf(\si_d(\al;\be))$.
The same holds for $\lf_{-\infty}$.
\end{lem}

\begin{proof}
By the definition of the Hausdorff measure and properties of $G'$, $\hf(\si_d(\al;\be)\cap G')=d^{\delta_*}\hf[(\al;\be)\cap G']$.
Noting that $d^{\delta_*}=d_*$ proves the first claim.
The second claim is just as simple.
\end{proof}

We can now show that $\lf_{-\infty}=\hf$.

\begin{thm}
  \label{t:lfinf}
  The measures $\lf_{-\infty}$ and $\hf$ on $\uc$ coincide.
\end{thm}

\begin{proof}
We want to show that $\psi_*\hf$,
the pushforward of $\hf$ under $\psi$, coincides with the Lebesgue measure $\lf_\infty=\psi_*\lf_{-\infty}$ on $\uc$.
Both measures $\psi_*\hf$ and $\lf_\infty$ are Borel probability measures
satisfying the invariance property of Lemma \ref{l:hf-inv}, and both have no atoms.
These facts are sufficient to prove that the two measures coincide.
Since the argument is standard, we leave the details to the reader.
\end{proof}

\subsection{Local Hausdorff measure estimates}
\label{ss:Haus-mes}
Below, we obtain upper bounds on the Hausdorff $\delta_*$-measures
of $G'\cap (\al;\be)$,
in terms of the measures $\lf_{-n}$.
Recall that $\Cc_{-n}$ is the union of level $-n$ connectors of $(f,K^*)$, and that $b_{-n}$
denotes the number of these connectors.
Recall that $c_0:=\frac{d-d_*}{d_*-1}$; see Lemma \ref{l:bn}. Finally, let $b_{-n}(X)$ be
the number of components of $X\cap \Cc_{-n}$, for $X\subset \uc$.

\begin{lem}
  \label{l:ncomp1}
Fix an arc $[\al;\be]\subset\uc$ and an integer $k\ge 0$.
The set $[\al;\be]\cap\Cc_{-n}$ with $n>l$ has at most $(b_0+c_0) d_*^{n+1}\lf_{-k}[\al;\be]+b_{-l}$
components where $l\ge 0$ (depending only on $\al$, $\be$, and $k$) is chosen so that
$d_*^{-1}\le d_*^l\lf_{-k}[\al;\be]<1$ if $\lf_{-k}[\al;\be]<1$ or $l=0$ if $\lf_{-k}[\al;\be]=1$.
\end{lem}

\begin{proof}
If $l=0$, then $1\le d_*\lf_{-k}[\al;\be]$ and, by Lemma \ref{l:bn},
$$
b_{-n}([\al;\be])\le b_{-n}<(b_0+c_0)d_*^n\le (b_0+c_0)d_*^{n+1}\lf_{-k}[\al;\be].
$$
Assume that $l\ne 0$.
Since
$$
|\psi_{-k}([\al; \be])|=\lf_{-k}[\al;\be]<d_*^{-l},
$$
then the map $\si_{d_*}^l$ is injective on $\psi_{-k}([\al; \be])$, hence,
by diagram (\ref{e:dia1}), the $l$-th iterate of $\si^\A$ is monotone on $(\al;\be)$ and the image of
 $[\al;\be]$ under this iterate is $[\si_d^l(\al); \si_d^l(\be)]$.
So, $d_*^{-1}\le \lf_{-k+l}([\si^l_d(\al); \si_d^l(\be)])$, and
 the same estimate as above applies with $k$, $n$, $\al$, $\be$ replaced,
 respectively, with $k-l$, $n-l$, $\si_d^l(\al)$, and $\si_d^l(\be)$.
Since $d_*^l\lf_{-k}[\al;\be]=\lf_{-k+l}[\si_d^l(\al);\si_d^l(\be)]$ then
$$b_{-n+l}((\si^\A)^l[\al;\be])\le (b_0+c_0)d_*^{n-l+1}\lf_{-k+l}((\si^\A)^l[\al;\be])$$
and it is easy to see that
$$(b_0+c_0)d_*^{n-l+1}\lf_{-k+l}((\si^\A)^l[\al;\be])=(b_0+c_0)d_*^{n+1}\lf_{-k}[\al;\be].$$
Since every component of
$\si_d^l([\al;\be]\cap\Cc_{-n})$ lies in some component of $(\si^\A)^l[\al;\be]\cap\Cc_{-n+l}$,
and
there are $b_{-l}$ holes of level $-l$ then
$$b_{-n}([\al;\be])\le b_{-l}+b_{-n+l}((\si^\A)^l[\al;\be]).$$
Therefore,
$$
b_{-n}([\al;\be])\le b_{-l} + (b_0+c_0)d_*^{n+1}\lf_{-k}([\al;\be]) .
$$
which implies the desired estimate.
\end{proof}

Lemma \ref{l:Hmes-Scn} uses Lemma  \ref{l:ncomp1}.

\begin{lem}
  \label{l:Hmes-Scn}
Fix an arc $[\al;\be]\subset\uc$ and an integer $k\ge 0$. Then the $\delta_*$-Hausdorff measure of $[\al;\be]\cap G'$
is at most $d_*(b_0+c_0)\la_0^{\delta_*} \lf_{-k}[\al;\be]$.
\end{lem}

\begin{proof}
Clearly, $[\al;\be]\cap G'\subset [\al;\be]\cap\Cc_{-n}$ for any $n>0$; by Lemma \ref{l:ncomp1} the latter set
consists of at most $(b_0+c_0) d_*^{n+1}\lf_k[\al;\be]+b_{-l}$ closed circle arcs each of which has length at most $\la_0 d^{-n}$
(by Lemma \ref{l:lan}). Since $(\la_0 d^{-n})^{\delta_*}=\la_0^{\delta_*}d_*^{-n}$,
for every $\e>0, n>0$ with $\la_0 d^{-n}<\e$ we get
$$
\Hmes_\e^{\delta_*}([\al;\be]\cap G')\le \la_0^{\delta_*} d_*^{-n} \bigl((b_0+c_0) d_*^{n+1}\lf_{-k}[\al;\be]+b_{-l} \bigr).
$$
Letting $n\to \infty$ (and remembering that $b_{-l}$ depends only on $\al$, $\be$, $k$ but not on $n$),
 we obtain the desired.
\end{proof}

Lemma \ref{l:Hmes-Scn} implies the following lower bound on $\lf_{-k}[\al;\be]$:
\begin{equation}\label{l:lowb}
\lf_{-k}[\al;\be]\ge \frac 1{d_*(b_0+c_0)\la_0^{\delta_*}} \Hmes^{\delta_*}([\al;\be]\cap G').
\end{equation}
It is important that this lower bound is independent of $k$!

\section{Geometric estimates}
\label{s:geom-est}
In this section, we consider certain conformal metrics on the plane and use
 them for extremal length / conformal moduli estimates.

\subsection{Conformal metrics}
\label{ss:metrA}
\def\grm{\mathrm{g}}
A (measurable) conformal (pseu\-do) metric on a Riemann surface $S$
 is an expression that looks locally, in any coordinate chart, as $\gf(z)=\mathrm{g}(z)|dz|$,
 where $z$ is a local coordinate, and $\grm$ is a nonnegative measurable function of $z$.
Transition from one local chart to another follows the standard transformation rules for $\grm$ and $dz$;
note that $\grm$ may vanish at certain points and even on certain regions.
One can multiply a conformal metric by a nonnegative function;
 the result is another conformal metric.
Given a piecewise smooth curve $\ga$ in $S$, the ($\gf$-)\emph{length} of $\ga$ is
\begin{equation}\label{e:len}
\Len_\gf(\ga):=\int_\ga \grm(z) |dz|=\int_{0}^{1} \grm(\ga(t)) |\dot\ga(t)| dt,
\end{equation}
where the last expression assumes that $\ga$ is
(piecewise $C^1$) parameterized by a parameter $t\in [0;1]$.
Also, if $D$ is any domain in $S$ covered by a single coordinate chart,
then the ($\gf$-)area of $D$ is defined as
\begin{equation}\label{e:area}
\area_\gf(D):=\int_D \grm(z)^2 |dx\wedge dy|,\quad z=x+iy.
\end{equation}
If $D$ is not covered by a single coordinate chart, then cut it into smaller subdomains,
each fitting to a coordinate chart, and sum up the corresponding areas.

\begin{dfn}[Extremal length]
\label{d:el}
Let $\Fc$ be a family of piecewise smooth curves in $S$ and $\gf$ be a
 conformal metric on $S$ of finite positive $\area_\gf(S)$.
Define the \emph{$\gf$-extremal length} of $\Fc$ as
\begin{equation}\label{e:elen}
\Lambda_{\gf}(\Fc)=\frac{\Len_{\gf}(\Fc)^2}{\area_\gf(S)}
\end{equation}
where $\Len_{\gf}(\Fc)$ is the infimum length of a curve from $\Fc$ with respect to $\gf$.
The \emph{extremal length} of $\Fc$ is $\el_S(\Fc)=\sup_{\gf}\Lambda_{\gf}(\Fc)$
where $\gf$ ranges through all measurable conformal metrics in $S$
of finite positive area.
If $S$ is fixed, it will be omitted from notation.
\end{dfn}

An \emph{(open) annulus} is a
 Riemann surface homeomorphic to the cylinder $\uc\times \R$.
Any annulus $A$ is conformally isomorphic to a \emph{unique} model open Euclidean
 cylinder $\uc\times (0;\mu)$ of circumference 1 (an isomorphism here is unique up to a rigid rotation);
 the height $\mu$ of this cylinder is called \emph{the modulus} of $A$
 and is denoted by $\mod(A)$.
Identify $A$ with its
 conformal model and let $\ol{A}$
be the corresponding closed cylinder $\uc\times [0;\mu]$.
Then one can talk about curves in $A$ that separate the two boundary circles
 $\uc\times\{0\}$ and $\uc\times\{\mu\}$ of $\ol{A}$.

The following is classical (see, e.g., \cite{ahl66}).

\begin{thm}
  \label{t:mod-above}
Let $A$ be an
 annulus, and define $\Fc_A$ as the
 family of all piecewise smooth closed curves in $A$ separating the two boundary components of $A$.
Then
$$
\mod(A)=\frac{1}{\el(\Fc_A)}\le 1/\Lambda_\gf(\Fc_A),
$$
for any conformal metric $\gf$ on $A$.
\end{thm}

Next, we go back to the basic setup with a degree $d$ complex
polynomial $f$ and a degree $d_*$ invariant component $K^*$ of $K(f)$. Recall that $g_f$ is the Green function for $f$.

\begin{dfn}[The metric $\ef$]
\label{d:ef}
Let $\ef$ be the smooth conformal metric on $\C\sm K(f)$ uniquely defined
by the property that $\ef=|dg_f|$ on any external segment.
On the basin of infinity, this is a flat (locally Euclidean) metric with singularities
at iterated preimages of the escaping critical points, that is,
at singularities of $g_f$.
We extend $\ef$ by zero to all of $\C$.
\end{dfn}

In a neighborhood of infinity, there is a holomorphic B\"ottcher coordinate $\phi_f$ that
conjugates $f$ with $z\mapsto z^d$; the Green function $g_f$ coincides with $\log|\phi_f|$
in this neighborhood of infinity.
Consider a polar coordinate system $(\rho,\ta)$ near infinity, where $\rho$ is the
distance to the origin and $\ta\in\R/\Z$ is the polar angle. In these polar coordinates
$\phi_f(z)=\rho_f e^{2\pi i\ta_f}$ which makes $\rho_f$ and $\ta_f$ into functions of $z$
so that $g_f(z)=\log \rho_f(z)$ and $\ta_f(z)$ is the external argument of $z$.
Functions $\rho_f$, $\ta_f$ embed a punctured neighborhood of infinity into
a Euclidean cylinder with the metric
\begin{equation}\label{e:eucylin}
|d\log\rho(z)|^2+|d(2\pi \ta)|^2=\frac{d\rho^2}{\rho^2}+4\pi^2 d\ta^2.
\end{equation}
Pulling back this metric from the cylinder to the dynamical plane of $f$, one obtains
the restriction of $\ef$ to a neighborhood of infinity.

Note that equipotentials
 that encircle all the critical points of $f$ have length $2\pi$ with respect to $\ef$.
More generally, one has the following straightforward lemma.

\begin{lem}
  \label{l:ef-circum}
An equipotential arc crossing the external rays $R_f(\ta)$
 for all $\ta\in [\al,\be]$ and for no other values of $\ta$, has $\ef$-length
 equal to $2\pi\lf_\infty[\al;\be]$.
Hence, the $\ef$-length of $O^*(t_n)$ is equal to $2\pi\mu_n$.
\end{lem}

Recall that $\lf_\infty$
is the Lebesgue probability measure on $\uc$,
and $\mu_n$ stands for the total length of $\Cc_n$, the union of all level $n$ connectors.

\subsection{Rectangles}
\label{ss:rect}
Recall that external rays are unbounded external intervals that extend from infinity to $K(f)$ or to a singularity of $g_f$.
By definition, an external ray does not contain singularities.
Let $\Sc_n$ be the intersection of $A_n$ with the union of all external rays
 that cross the inner boundary $O^*(t_n)$ of $A_n$.
The set $\Sc_n$ is an open set and every component of it is foliated by segments of external rays;
 each such segment connects the two boundary components of $A_n$.
Indeed, since, by definition, an external ray contains no singularities of $g_f$,
 some neighborhood of it is foliated by nearby external rays.

\begin{dfn}[Rectangles]\label{d:re1}
Components of $\Sc_n$ are called \emph{level $n$ open rectangles},
and their closures are called \emph{level $n$ closed rectangles}.
\end{dfn}

Note that some broken rays that reach the inner boundary $O^*(t_n)$ of $A_n$
 may not lie in closed rectangles.
Indeed, suppose there are more than $2$ broken rays that meet at a singularity
 of $g_f$ in $A_n$, for some $n$.
Then 
 two of them, denoted here by $R'$ and $R''$, such that
$(R'\cup R'')\cap A_n$
 forms the shape of a letter $Y$, are contained in closed rectangles.
However the remaining broken rays (located between $R'$ and $R''$) are partially isolated
and are not contained in the union of all closed rectangles.
Note also that a smooth segment of a broken ray may separate two adjacent closed rectangles;
 even though this piece has no singularities in $A_n$, it extends to a bounded external segment eventually hitting
 a singularity in the direction of growing potential.
The inner boundary of $A_n$ is contained in the union of closed rectangles,
 as follows immediately from the definitions.

With respect to the metric $\ef$, each level $n$ open rectangle is
a Euclidean rectangle of height $t_{n+1}-t_n$.
In each open rectangle, there is a preferred choice of \emph{vertical and  horizontal} directions.
Namely, the vertical direction is the one along external rays,
and the horizontal direction is the one along equipotential curves.
Edges of a rectangle can also be classified as vertical or horizontal.
Vertical edges are segments of broken rays, e.g.,
they may contain parts of external intervals and singularities of $g_f$.
Horizontal edges lie in equipotential curves.

The following lemma describes how two closed level $n$ rectangles may touch.
Recall that both $\R/\Z$ and $\R_{>0}$ come with natural orientations.
For this reason, it makes sense to talk of horizontal direction to the right and
 and of vertical direction upwards.

\begin{lem}
  \label{l:er-touch}
Let $Q_1$, $Q_2$ be distinct closed level $n$ rectangles.
Either $Q_1\cap Q_2=\0$, or $Q_1\cap Q_2$ is a subset of the vertical edges of both $Q_1$ and $Q_2$.
If the right edge $E_1$ of $Q_1$ has points in common with the left edge $E_2$ of $Q_2$,
 then either $E_1=E_2$ or $E_1\cap E_2$ is a bottom subsegment of both $E_1$, $E_2$,
 whose top endpoint is a singularity of $g_f$.
\end{lem}

\begin{proof}
As the first claim of the lemma is clear, we prove only the last claim.
Let $I$ be the maximal by inclusion bottom subsegment of $E_1\cap E_2$,
 and let $\om$ be its top endpoint.
Since the two edges $E_1$, $E_2$ diverge from $\om$ upwards, that is,
 form a $Y$-shape, the point $\om$ is necessarily a singularity of $g_f$.
It is impossible that $E_1$, $E_2$ merge again above $\om$, as otherwise
 the topological hull of $E_1\cup E_2$ would contain an open bounded disk free from external rays.

Note that $\om$ may be a point on the inner boundary of $A_n$,
 in which case $Q_1$ and $Q_2$ have only one point in common (this point is $\om$).
\end{proof}

Say that two closed rectangles $Q_1$ and $Q_2$ as
 in Lemma \ref{l:er-touch} are \emph{$Y$-attached} to each other
 if (possibly after relabeling $Q_1$ and $Q_2$) the right edge $E_1$ of $Q_1$
 and the left edge $E_2$ of $Q_2$ are such that $E_1\cap E_2$ is a proper bottom
 subsegment of both.
Let $\om$ be the top endpoint of this segment; it is a singularity of $g_f$
 called the \emph{branch point of $Q_1\cup Q_2$}.
Clearly, there is a bubble attached to $\om$ and otherwise disjoint from $Q_1\cup Q_2$.

\begin{lem}
\label{l:Scn}
Every closed level $n$ rectangle is mapped by $f$ into (not necessarily onto)
a closed level $n+1$ rectangle. If closed level $n$ rectangles $Q_1$, $Q_2$ are $Y$-attached and
mapped to the same level $n+1$ closed rectangle, then the
branch point of $Q_1\cup Q_2$ is critical.
\end{lem}

\begin{proof}
The first claim of the lemma is evident.
Assume that $Q_1$, $Q_2$ are closed level $n$ rectangles that are $Y$-attached.
Let $\om$ be the branch point of $Q_1\cup Q_2$.
If $\om$ is not a critical point of $f$, then $f$ is a local homeomorphism near $\om$,
 which implies that $f(Q_1)$, $f(Q_2)$ are also $Y$-attached, hence not
 in the same closed level $n+1$ rectangle.
\end{proof}

We now define a convenient conformal metric $\qf$, which will later be used for
estimating the moduli of fundamental annuli.
Recall that, by Lemma \ref{l:ef-circum},
an equipotential arc that crosses the external rays $R_f(\ta)$
for all $\ta\in [\al,\be]$ and for no other values of $\ta$, has $\ef$-length
equal to $2\pi\lf_\infty[\al;\be]$.
Hence, the $\ef$-length of every $O^*(t_n)$ is $2\pi\mu_n$.

\subsection{The metric $\qf$}
\label{ss:pf}
Define the metric $\qf$ on each annulus $A_n$ as the
metric $\ef$ times the \emph{indicator function} of $\Sc_n$
(equal to $1$ on $\Sc_n$ and to $0$ elsewhere)
times a constant factor $(2\pi\mu_n)^{-1}$ chosen so that the total width of
all rectangles from $\Sc_n$ equals $1$ (equivalently, $\Len_\qf(O(t))=1$
 for all $t\in (t_n;t_{n+1})$).
To define $\qf$ on $O^*(t_n)$, for all $n\in\Z$,
set $\qf|_{O^*(t_n)}$ to be $\ef$ times a large factor so that the
$\qf$-length of any segment of $O^*(t_n)$ is larger than the length of
nearby paths with the same endpoints otherwise lying in $A_n$ or $A_{n+1}$.

\begin{rem}
Informally, the definition of $\qf$ is based on two principles:
 1) the metric is nonzero only on rectangles,
 2) it is normalized so that to make the length of $O(t)$ equal 1, for
 nonsingular $O(t)$.
External rays that avoid level $-n<0$ rectangles but not level $-n+1$ rectangles can be bypassed by
 a curve encircling $K^*$; the bigger $n$ is, the harder it is to bypass.
For this reason, we increase the cost of crossing level $n$ rectangles with $n$.
On the other hand, a curve around $K^*$ is allowed to cross bubbles ``for free''
 (that is, with zero length added),
 since crossing bubbles does not contribute to the task of going around $K^*$.
Nonsingular equipotential ovals around $K^*$ have the same length, no matter
 how close to $K^*$ they go.
All this should convince the reader that considering the metric $\qf$ is
 reasonable when dealing with the external length of the family
 of piecewise smooth curves in $U^*$ going once around $K^*$.
\end{rem}

By Definition \ref{d:psiex}, $\psi_n:\Cc_n\to\uc$ maps the closures of all complementary components of $\Cc_n$ to points.
Also, every point of $\Sc_n$ has a unique external argument.
If $\ta$ is the external argument of a point $z\in\Sc_n$, then $\ta\in\Cc_n$,
and we set $\Psi_n(z):=\psi_n(\ta)$. Note that $\Psi_n$ maps $\Sc_n$ to $\uc$,
depends continuously on $z\in\Sc_n$, and can be extended
to $\ol\Sc_n$ by continuity.
Clearly, $\Psi_n$ maps any equipotential arc $I\subset \Sc_n$ to an arc $\Psi_n(I)\subset \uc$
with $\qf(I)=\lf_\infty(\Psi_n(I))$ (where $\lf_\infty$ is the Lebesgue measure on $\uc$).
Note also that $\Psi_n$ is constant on external intervals.

We will use the following important relationship between the metric $\qf$ on $A_n$ and
 the measure $\lf_n$ on $\uc$.
Under our adopted identification between $\uc$ and $O(t_{m+1})$, the $\lf_n$-measure
 of any equipotential arc $I$ in $O(t_{m+1})$ equals $\qf(I')$, where $I'\subset O(t)$ for some $t\in (t_n;t_{n+1})$
 consists of all points, where $O(t)$ intersects the external rays through $I$.

\begin{lem}
\label{l:qf-len}
Consider a piecewise smooth path $\gamma:[0;1]\to \ol A_n$.
The $\qf$-length of $\ga$ is bounded below by the length of $\Psi_n\circ\ga$,
i.e., by the length of the set of all points $\Psi_n(\ga(t))$, where $t\in [0;1]$
is such that $\ga(t)\in\Sc_n$. Thus, the $\qf$-length of a rectifiable path $\ga:[0;1]\to \ol A_n$ is bounded below
 by the $\lf_n$-measure of the set $\Arg(\ga[0;1])$.
\end{lem}

\begin{proof}
The lemma follows from
the fact that the Euclidean length of a path is bounded below by
the length of its projection to the $x$-axis.
\end{proof}

\subsection{A lower bound on the modulus}
\label{ss:lowb}
A lower bound on the modulus $\mod(U^*\sm K^*)$ can be obtained if we estimate the modulus of $A_0$,
 see Corollary \ref{c:Gro}.
Recall that $\mu_0$ is the total length of $\Cc_0$,
 equivalently, the total width of $\Sc_0$ with respect to the metric $\ef$.
Also, $\eps$ stands for the value of $g_f$ on $O^*(t_0)=\d U^*$.

\begin{lem}
\label{l:modA0}
The modulus of $A_0$ is bounded below by $\frac{(d-1)\eps}{2\pi d_*\mu_0}$.
\end{lem}

\begin{proof}
Let $t_{0,0}=t_{0}<t_{0,1}<\dots <t_{0,l}=t_1$
 be a subdivision of the interval $[t_{0};t_1]$ formed by all $t$ such that $O(t)$
 contains singularities of $g_f$.
Set $A^*_{0,k}$, where $0\le k<l$, to be the annulus between $O(t_{0,k})$ and $O^*(t_{0,k+1})$.
Note that $A^*_{0,k}$ does not include the bubbles attached to $O^*(t_{0,k})$.
Then the metric $\ef$ endows the annulus $A^*_{0,k}$ with a structure of a Euclidean
 cylinder of height $t_{0,k+1}-t_{0,k}$ and of circumference $\Len_\ef(O(t_{0,k}))$.
Firstly, by the Gr\"{o}tzsch inequality, one has
$$
\mod(A_0)\ge \sum_{k=0}^{l-1} \mod(A^*_{0,k})=\sum_{k=0}^{l-1}\frac{t_{0,k+1}-t_{0,k}}{\Len_\ef(O(t_{0,k}))}.
$$
Secondly, the $\ef$-circumference of the cylinder $A^*_{0,k}$ can be estimated above as $2\pi\mu_1=2\pi d_*\mu_0$,
 by Lemma \ref{l:ef-circum}.
It follows that
$$
\mod(A_0)\ge \frac{t_1-t_{0}}{2\pi\mu_1}=\frac{(d-1)\eps}{2\pi d_*\mu_0},
$$
 as claimed.
\end{proof}

Combining Lemma \ref{l:modA0} with Corollary \ref{c:Gro}, we obtain a
 lower bound on $\mod(U^*\sm K^*)$:
\begin{equation}
\label{e:low}
\mod(U^*\sm K^*)\ge \frac{1}{d_*-1}\mod(A_0)\ge \frac{(d-1)}{2\pi d_*(d_*-1)}\cdot\frac{\eps}{\mu_0}.
\end{equation}
The lower bound stated in Theorem \ref{t:main-an} can now be deduced using an upper bound on $\mu_0$.

\begin{proof}[Proof of Theorem \ref{t:main-an}, lower bound]
By Lemma \ref{l:mun}, we have $\mu_0\le (d-1)^{m+1}d^{-m-1}$.
Plugging this into (\ref{e:low}) yields the desired estimate.
\end{proof}

The next example uses the lower bound of Theorem \ref{t:main-an}.

\begin{ex} There exists a sequence $f_m$ of degree $4$ polynomials with disconnected filled Julia sets $K_m$
and an invariant quadratic-like component $K^*_m$ of $K_m$ with the following properties:
\begin{enumerate}
  \item the invariant component $K^*_m$ of $K(f_m)$ has degree 2;
  \item there are two escaping critical points $\om_m^\pm$ of $f_m$;
  \item we have $f_m^m(\om_m^-)=\om_m^+$;
  \item the point $\om_m^-$ is not in a bubble;
  \item $g_f(\om_m^-)=\eps_m>0$;
  \item $\eps_m\to 0$ while $\eps_m(4/3)^m\to \infty$.
\end{enumerate}
For these polynomials, $\mod(U_m^*\sm K_m^*)\ge \frac{3}{2\pi}\cdot \eps_m(4/3)^m\to \infty$
while $g_{f_m}(\om_{-,m})=\eps_m\to 0$ as $m\to\infty$.

Indeed, such polynomials $f_m$ exist as follows from \cite[Theorem 1.2]{DM08}.
Note that $m=\max(\chi_m)$, where $(\A_m,\chi_m)$ is the combinatorial gap data of $f_m$,
 which is consistent with our earlier use of $m$.
To lighten the notation we now fix $m$ and mostly do not use it as a subscript.
Recall that $\mu_n=\Len_\ef(O^*(t_n))$.
By definition of $t_n$, the critical point $\om_-$ lies in $O^*(t_0)$,
while $\om_+$ lies in $O^*(4^mt_0)$.
The map $f:O^*(t_n)\to O^*(t_{n+1})$ is a covering of degree $d_n$,
where $d_n=d-\sum_{k\ge n} a_k$, cf. Section \ref{ss:holconn}.
In particular, $d_n=d_*=2$ for all $n\le 0$, and $d_n=d=4$ for all $n>m$.
For $n$ in the range $(0,m]$, the degree $d_n$ equals 3.

It is easy to compute $\mu_n$ for all $n$ using the fact that $d\mu_n=d_n\mu_{n+1}$
and that $\mu_n=1$ for all $n>m$; the formula for $\mu_n$ is independent of $\eps$, and,
in particular, $\mu_0=\frac 23(\frac 34)^m$.
By our choice of $\eps_m$, while $g_{f_m}(\om_{-,m})=\eps_m\to 0$, we, on the other hand,
see that, by equation \eqref{e:low},
$$
\mod(U_m^*\sm K_m^*)\ge \frac{(d-1)}{2\pi d_*(d_*-1)}\cdot\frac{\eps_m}{\mu_0}=\frac{3}{2\pi}\cdot \eps_m(4/3)^m\to \infty.
$$
\end{ex}

\subsection{Length and area estimates}
\label{ss:len-qf}
In the rest of Section \ref{s:geom-est}, we work with the metric $\qf$
supported on the closure of the union of $\Sc_n$ over $n\in\Z$.
Let $\Fc_\circ$ be the family of all piecewise smooth simple closed curves in $U^*\sm K^*$
that separate $K^*$ from the boundary of $U^*$.
By Theorem \ref{t:mod-above}, one has $\mod(U^*\sm A^*)\le 1/\Lambda_\qf(\Fc_\circ)$,
and, therefore, it suffices to estimate $\Lambda_\qf(\Fc_\circ)$ below.
Write $G$ for the renormalization gap of $(f,K^*)$ (see Definition \ref{d:renormG}).
Recall that $G':=G\cap\uc$ and that $b_0$ is the number of
components in $\Cc_0$.
Also, recall that $R^*(\al)=R^*_f(\al)$ is a ray towards $K^*$
 which may be a part of $R(\al)$ terminating on the boundary of some bubble,
 or it may be a concatenation of $R(\al)$ with some singularities of $g_f$ and segments of
 some external intervals so that $R^*(\al)$ extends towards $K^*$ ``as much as possible''.

The following lemma uses the identification between $\uc$ and $O(t_{m+1})$
 to reinterpret an arc $[\al;\be]$ of $\uc$ as an equipotential arc in $O(t_{m+1})$

\begin{lem}
  \label{l:len-qf}
Suppose that an arc $\gamma$ connects the rays $R^*(\al)$ and $R^*(\be)$ and
 is homotopic to $[\al;\be]$, where the homotopy is allowed to slide the endpoints of the arc
 along $R^*(\al)$ and $R^*(\be)$ but not to move them otherwise.
Then the $\qf$-length of $\gamma$ is bounded below by
$$
\frac 1{d_*(b_0+c_0)\la_0^{\delta_*}} \Hmes^{\delta_*}(G'\cap [\al;\be]),\quad \hbox{where}\quad c_0=\frac{d-d_*}{d_*-1}.
$$
It follows that the $\qf$-length of any curve $\Gamma\in\Fc_\circ$
 satisfies
$$
\mathrm{Len}_\qf(\Gamma)\ge \frac 1{d_*(b_0+c_0)\la_0^{\delta_*}} \Hmes^{\delta_*}(G').
$$
\end{lem}

\begin{proof}
By definition of the metric $\qf$, one can replace $\ga$ with
 a homotopic (rel. the endpoints) arc,
 which is not longer than $\ga$,
 has at most finitely many intersection points with the union of $O^*(t_n)$ for $n\in\Z$,
and intersects every external ray at most finitely many times.
Assume that $\ga$ itself satisfies these properties.
The intersection $\ga\cap\bigcup_{n\in\Z}\Sc_n$ consists of at most countably
many disjoint arcs $I_j$ of $\ga$, each of which is a component of $\ga\cap\Sc_n$ for some $n$,
such that the union of the corresponding sets $\Theta_j:=\Arg(I_j)$ taken over all such $j$ covers
all of $G'\cap [\al;\be]$, except at most countably many points.

Fix $j$ as above associated with $\Sc_n$.
By Lemma \ref{l:qf-len}, $\Len_\qf(I_j)\ge \lf_n(\Theta_j)$.
On the other hand, by Lemma \ref{l:Hmes-Scn},
$$
\lf_n(\Theta_j)\ge \frac 1{d_*(b_0+c_0)\la_0^{\delta_*}} \Hmes^{\delta_*}(\Theta_j\cap G').
$$
Clearly, its right-hand side is independent of $n$
and depends additively on $\Theta_j$. Since the sets $\Theta_j:=\Arg(I_j)$ cover almost all of $G'$,
and by the additivity, the total length $\Len_\qf(\ga)$
has the claimed lower bound.
\end{proof}

We also need an upper bound on the $\qf$-area of $U^*\sm K^*$.
Recall that $\mu_n=\lf_\infty(\Cc_n)$ (see Section \ref{ss:holconn}).

\begin{lem}
  \label{l:area-qf}
The $\qf$-area of $\Sc_n$ is equal to $(2\pi\mu_0)^{-1}(d/d_*)^{n}(d-1)\eps$.
It follows that the $\qf$-area of $U^*\sm K^*$ is equal to $\frac{d_*}{d-d_*}\cdot \frac{(d-1)\eps}{2\pi\mu_0}$.
\end{lem}

\begin{proof}
From the viewpoint of the metric $\ef$, the set $\Sc_n$ consists of several rectangles of
height $t_{n+1}-t_n$ that have total width $2\pi\mu_n$.
On $\Sc_n$, the metric $\qf$ is just a constant multiple of $\ef$,
normalized to make the total width of $\Sc_n$ equal $1$.
Hence, $\qf=(2\pi\mu_n)^{-1}\ef$ on $\Sc_n$.
Now, the $\ef$-area of $\Sc_n$ is $2\pi\mu_n(t_{n+1}-t_n)$, which is the total width times the height,
and the $\qf$-area is therefore
$$
\area_\qf(\Sc_n)=\frac{1}{(2\pi\mu_n)^2}\area_\ef(\Sc_n)=\frac{t_{n+1}-t_n}{2\pi\mu_n}.
$$
It remains to recall that $t_n=d^n\eps$ and $\mu_n=d_*^n\mu_0$ for $n\le 0$;
the desired formula for the area of $\Sc_n$ follows.

The last claim of the lemma follows from the first claim by
 taking the sum over all $n<0$; this sum is a geometric series.
\end{proof}

\subsection{Moduli and extremal distance estimates}
\label{ss:upper}
Using Lemmas \ref{l:len-qf} and \ref{l:area-qf}, one can write
$$
\Lambda_\qf(\Fc_\circ)\ge \frac{2\pi\mu_0(d-d_*)}{d_*^3(d-1)(b_0+c_0)^2\la_0^{2\delta_*}\eps}\Hmes^{\delta_*}(G')^2
$$
Theorem \ref{t:mod-above} now implies the following.

\begin{thm}
\label{t:upper}
The modulus of $U^*\sm K^*$ satisfies the inequality
$$
\mod(U^*\sm K^*)\le\frac{d_*^3(d-1)(b_0+c_0)^2\la_0^{2\delta_*}\eps}{2\pi\mu_0(d-d_*)\Hmes^{\delta_*}(G')^2},
$$
\end{thm}

It remains to estimate several values on the right-hand side.

\begin{proof}[Proof of Theorem \ref{t:main-an}]
The lower bound on the modulus of $U^*\sm K^*$ was already established in Section \ref{ss:lowb}.
For the upper bound, we use the inequality from Theorem \ref{t:upper}
 together with Lemmas \ref{l:bn0},
 \ref{l:lan}, \ref{l:mun}, and Theorem \ref{t:HdimG}.
Namely, $b_0\le (d-d_*)(d-1)^m$, by Lemma \ref{l:bn0}, and
\begin{equation}
\label{e:b0+c0}
b_0+c_0\le (d-d_*)((d-1)^m+1)\le (d-d_*)d^m,
\end{equation}
 by (\ref{e:c0}).
Here, we used the assumption that $m>0$.
Also, $\la_0\le d^{-m}$, by Lemma \ref{l:lan}, and
$$
\mu_0\ge \frac{d_*^{m+1}}{d^{m+1}},\qquad \Hmes^{\delta_*}(G')\ge \frac{1}{d_*^2},
$$
 by Lemmas \ref{l:mun} and Theorem \ref{t:HdimG}, part (3).
Substituting all these inequalities to the expression on the right-hand side of the inequality from Theorem \ref{t:upper},
 we obtain, after straightforward simplifications, that
$$
\mod(U^*\sm K^*)\le \frac{d^{3m+2}}{d_*^{3m-2}}\cdot \frac{(d-d_*)\eps}{2\pi\cdot\Hmes^{\delta_*}(G')^2}\le
\frac{d^{3m+2}(d-d_*)\eps}{2\pi\cdot d_*^{3m-6}},
$$
which completes the proof of Theorem \ref{t:main-an}.
\end{proof}

The proof of Theorem \ref{t:m=0} uses the same lemmas but the estimates
 in the case $m=0$ are a bit different.

\begin{proof}[Proof of Theorem \ref{t:m=0}]
We just indicate the differences with the proof of Theorem \ref{t:main-an}.
An upper bound on $b_0+c_0$
 is replaced with another one, $b_0+c_0\le c_0d_*$, valid only for $m=0$ (see Lemma \ref{l:bn}).
For $\la_0$, we use the upper bound $\la_0\le d_*d^{-1}$ from Lemma \ref{l:lan}.
Also, we keep the other estimates in their original form, without further relaxations.
Computational details are straightforward; for this reason, they are omitted here.
\end{proof}

Local versions of the moduli estimates deal with
 the extremal distance $\mathrm{ED}(\al;\be)$ between two (smooth or) broken rays to $K^*$
 with arguments $\al$ and $\be$ (Theorem \ref{t:ED}); these local estimates are based on
 the same techniques as above.

\begin{proof}[Proof of Theorem \ref{t:ED}]
Set $\Omega:=U^*\cap \Sigma(\al;\be)$ and $\Fc:=\Fc_\Omega(A,B)$.
Also, to shorten the notation, we temporarily denote $\Hmes^{\delta_*}(G'\cap [\al;\be])$ by $H$,
 and use the convention that $C_1$, $C_2$, $\dots$ stand for functions of $d$, $d_*$, $m$ that do not
 depend on $\al$, $\be$, or $\eps$.
By definition of extremal length (Definition \ref{d:el}), it suffices to estimate $\Lambda_\qf(\Fc)$ from below.
Firstly, the $\qf$-length of any $\ga\in\Fc$ satisfies
$\Len_{\qf}(\ga)\ge  C_1 H$, by Lemma \ref{l:len-qf}, where
$$
C_1=\frac 1{d_*(b_0+c_0)\la_0^{\delta_*}}\ge \begin{cases}
                                                \frac{d_*^{m-1}}{d^m(d-d_*)}, & \mbox{if } m>0,\\
                                               \frac {d_*-1}{(d-d_*) d_*^{1+\delta_*}}, & \mbox{if } m=0.
                                             \end{cases}
$$
On the other hand, the $\qf$-area of $\Omega$ is estimated above as $\area_\qf(U^*\sm K^*)$,
 and the latter is $C_2\eps$ by Lemma \ref{l:area-qf}, where
$$
C_2=\frac{d_*(d-1)}{2\pi\mu_0(d-d_*)}\le \frac{d^{m+1}(d-1)}{2\pi d_*^m (d-d_*)}.
$$
It follows that
$$
\mathrm{ED}(\al;\be)\ge \Lambda_\qf(\Fc)\ge  \frac{(C_1H)^2}{C_2\eps}=\frac{C_3H^2}{\eps},\quad\hbox{where}\quad C_3:=\frac{C_1^2}{C_2},
$$
as claimed.
The expression for $C$ given after Theorem \ref{t:ED} is obtained immediately from the above estimates on $C_1$ and $C_2$.
\end{proof}

Recall that the case $d=3$ and $d_*=2$ is among the most interesting; $m=0$ automatically.
In this case, $f:U^*\to f(U^*)$ is a quadratic-like restriction of the cubic map $f$.
All the estimates become concrete in this situation, for example, $\mod(U^*\sm K^*)$
 is between $0.2\eps$ and $147\eps$, as follows easily from Theorem \ref{t:m=0}
 (substituting $d=3$, $d_*=2$, and $m=0$).
The inequality from Theorem \ref{t:ED} is true with $C=0.1$.


\end{document}